\documentclass[11pt]{article}
\title{\textsc{Percolations on random maps I: \\ half-plane models}}

\author{Omer Angel\footnote{University of British Columbia} \and
  Nicolas Curien\footnote{CNRS and University Paris 6}}

\date{}

\usepackage{amsmath,amsfonts,amssymb,amsthm,mathrsfs}
\usepackage[a4paper,vmargin={3.5cm,3.5cm},hmargin={2.5cm,2.5cm}]{geometry}
\usepackage{graphicx}
\usepackage{pstricks}
\usepackage[font=sf, labelfont={sf,bf}, margin=1cm]{caption}
\usepackage[colorlinks=true,linkcolor=black,citecolor=black,urlcolor=blue,pdfborder={0 0 0}]{hyperref}

\usepackage{enumerate}
\usepackage{cleveref}
  \crefname{theorem}{Theorem}{Theorems}
  \crefname{lemma}{Lemma}{Lemmas}
  \crefname{remark}{Remark}{Remarks}
  \crefname{proposition}{Proposition}{Propositions}
  \crefname{definition}{Definition}{Definitions}
  \crefname{corollary}{Corollary}{Corollaries}
  \crefname{section}{Section}{Sections}
  \crefname{figure}{Figure}{Figures}

\newcommand{\one}{{\triangle^{1}}}
\newcommand{\two}{{\triangle^{2}}}
\newcommand{\total}{\one,\two}
\newcommand{\allthree}{\{\one,\two,\square\}}

\newtheorem{theorem}{Theorem}[]
\newtheorem{definition}{Definition}[]
\newtheorem{proposition}[theorem]{Proposition}
\newtheorem{lemma}[theorem]{Lemma}

\theoremstyle{definition}
\newtheorem*{remark}{Remark}

\renewcommand{\P}{\mathbb P}
\newcommand{\E}{\mathbb E}
\newcommand{\Z}{\mathbb Z}
\newcommand\site{\mathrm{site}}
\newcommand\face{\mathrm{face}}
\newcommand\bond{\mathrm{bond}}
\newcommand\bM{\mathbf M}
\newcommand\cM{\mathcal M}
\newcommand\cE{\mathcal E}
\newcommand\cR{\mathcal R}
\newcommand\cL{\mathcal L}
\newcommand\cC{\mathcal C}
\newcommand\cH{\mathcal H}
\newcommand\cS{\mathcal S}

\begin{document}

\maketitle

\begin{abstract}
  We study Bernoulli percolations on random lattices of the half-plane
  obtained as local limit of uniform planar triangulations or
  quadrangulations. Using the characteristic spatial Markov property or
  peeling process \cite{Ang03} of these random lattices we prove a
  surprisingly simple universal formula for the critical threshold for bond
  and face percolations on these graphs.  Our techniques also permit us to
  compute off-critical and critical exponents related to percolation
  clusters such as the volume and the perimeter.
\end{abstract}


\section{Introduction}

In this work we study different types of percolations (bond and site on the
graph and its dual) on several types of infinite random maps.  For the sake
of clarity we focus on three kinds of maps: triangulations, two-connected
triangulations and quadrangulations, though our method is more general and
we shall indicate this at times.  We show that the spatial Markov property
of the underlying random lattice can be used as in \cite{Ang03} in order to
compute the critical threshold for percolation as well as geometric
properties of critical and near critical clusters.  In order to state our
results precisely let us start by introducing rigorously the random
lattices which we are working with.

\paragraph{Random infinite lattices.} 
Recall that a finite planar map (map in short) is a finite connected graph
embedded in the two-dimensional sphere seen up to continuous deformations.
The last decade has seen the emergence and the development of the
mathematical theory of ``random planar maps''.  A primary goal of this
theory is to understand the geometry of large random planar structures.

One fruitful approach consists of defining infinite random maps which are
the so-called \emph{local limits} of random planar maps and studying their
properties.  This idea has been first introduced in the seminal work of
Benjamini \& Schramm \cite{BS01} in the context of planar maps and is also
related to works of Aldous on local limits of trees \cite{AS04}.  Let us
present this setup.  As usual in the context of planar maps, we work with
rooted maps, meaning maps with a distinguished oriented edge $\vec{e}$
called the {\bf root edge} of the map.  The origin vertex of the root edge
is called the origin or root vertex of the map.  Following \cite{BS01} we
define a topology on the set of finite maps: If $m,m'$ are two rooted maps,
the local distance between $m$ and $m'$ is
\[
\mathrm{d_{loc}}(m,m') = \big(1+R)^{-1},
\]
where $R$ is the maximal radius so that $B_R(m)$ is isomorphic to
$B_R(m')$.  Here, $B_r(m)$ is the ball of radius $R$ in $m$ around the
origin, namely the map formed by the edges and vertices of $m$ that  are at
graph distance smaller than or equal to $R$ from the origin.
The set of finite maps is not complete for this metric and so we shall work
in its completion which also includes  infinite maps (see \cite{CMMinfini}
for a detailed exposition and references).

In this work, we focus on two specific kinds of planar maps:
triangulations (all faces have degree $3$) and quadrangulations (all faces
have degree $4$).  We also split the set of triangulations according to
their connectivity properties: A $1$-connected triangulation is just a
(connected) triangulation and a $2$-connected triangulation is a
triangulation with no cut-vertex.  It is easy to see that a triangulation
can only fail to be $2$-connected if some vertex has a self loop (an edge
whose target and origin vertices are confounded).

In the following, all the quantities referring to $1$ or $2$-connected
triangulations are denoted with the symbols $\one, \two$, and the ones
referring to quadrangulations are denoted with the symbol $\square$.  To
make statements that hold simultaneously about various types of maps we
shall use $*$'s to indicate one of those, or possibly some other type of
planar map (since our methods work in much greater generality).

 
We review the now classical construction of the Uniform Infinite Planar
Maps as weak local limits w.r.t.\ $\mathrm{d_{loc}}$ of uniform finite
maps.  Let $* \in \{\total, \square\}$ and for $n\geq 0$ we write $M_n^{*}$ for a
random variable uniformly distributed over the set of type-$*$ maps with
$n$ vertices. Then we have the following convergence in distribution for
$\mathrm{d_{loc}}$
\begin{equation}
  M_n^{*}  \xrightarrow[n\to\infty]{(d)}  M_\infty^{*}.   \label{UIPM}
\end{equation}
The object $M_\infty^{*}$ is a random infinite rooted planar map called the
($1$ or $2$-connected) Uniform Infinite Planar Triangulation (UIPT) if $ *
\in \{ \total\}$ and the Uniform Infinite Planar Quadrangulation (UIPQ) if
$* = \square$.  The convergence \eqref{UIPM} was established by Angel \&
Schramm \cite{AS03} in the triangular case $* \in \{\total\}$ and by Krikun
\cite{Kri05} in the quadrangulation case $*=\square$.  The UIPQ has also
been constructed by other means by Chassaing \& Durhuus \cite{CD06}, see
also \cite{CMMinfini}.  These random lattices have attracted a great deal
of attention in recent years \cite{Ang03, BCsubdiffusive, GGN12, Kri08,
  LGM10}.  Their large scale geometry is still a source of intensive
research and is tightly connected (see \cite{CLGMplane}) to the Brownian
map --- the universal continuous random surface obtained as the scaling
limit of properly renormalized random planar maps --- studied by Le Gall
and by Miermont \cite{LG11,Mie11}.
 
An important area of research on the random lattices ${M}_{\infty}^*$ is to
understand the behavior of statistical mechanics models on them.  Angel
\cite{Ang03} already studied site percolation on the UIPT and in particular
proved that the critical percolation threshold is almost surely
$\frac{1}{2}$.  In this paper, we extend this analysis to several other
types of percolation, including both bond and site on both the map and its
plane dual.

We pursue the analysis of percolation on random maps by focusing on
half-planar models.  These models indeed have an especially useful spatial
Markov property which makes the analysis of the percolation process much
simpler (see \cite{AR13} for a study of this property).  These pages can
thus be seen as a step towards the analysis of percolations on the
full-plane UIP$*$, which we do in a subsequent paper \cite{ACpercopeel2}.

In order to construct these half-plane models we first extend \eqref{UIPM}
to maps with a \emph{boundary}: A triangulation (or a quadrangulation) with
a \emph{boundary} is a planar map whose faces are triangles
(resp.\,quadrangles) except the face incident on the right of the
distinguished oriented edge which can be of arbitrary degree. This face is
called the external face. The perimeter $\partial m$ of a map $m$ with a
boundary is the degree of the external face.
In general, the boundary of a map $m$ can possess ``pinch-points'', that
are vertices visited at least twice during the contour of the external
face.  If the boundary does not have pinch-points we say that the boundary
is {\em simple}, or that $m$ is simple.

\begin{figure}[!h]
  \begin{center}
    \includegraphics[width=10cm]{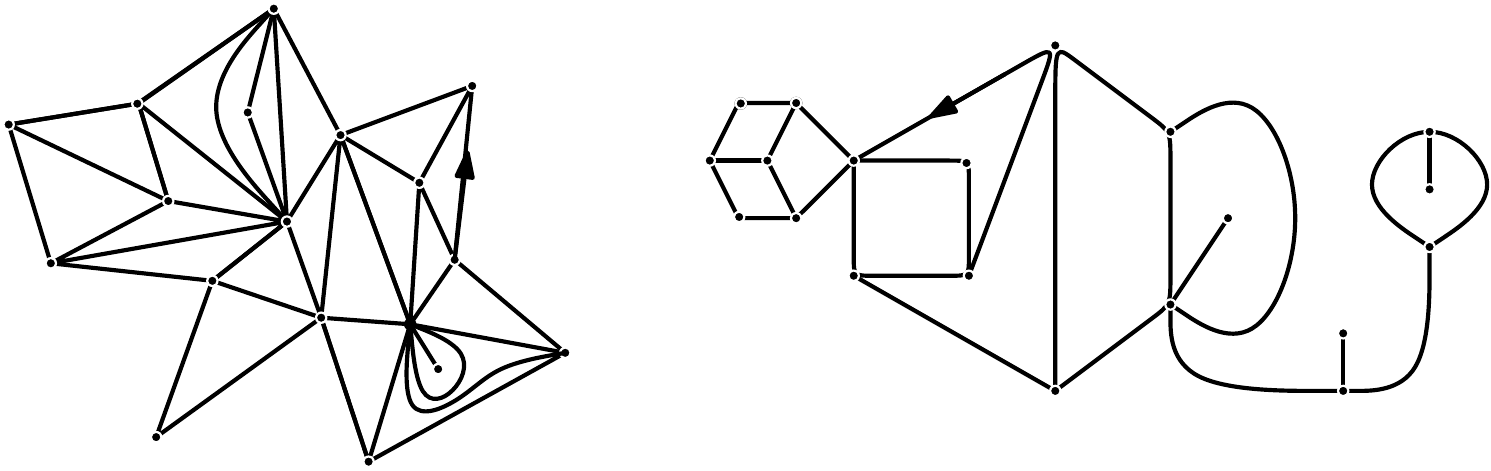}
    \caption{A type $1$ triangulation (note the triangle inside the
      self-loop) with simple boundary and a quadrangulation with general
      boundary.}
  \end{center}
\end{figure}
 
In the following, all the maps with a boundary that we consider are simple
and a map with simple boundary of perimeter $p\geq 1$ is also called a map of
the $p$-gon.  For $n \geq 0$ and $p \geq 1$, we denote by
$\mathcal{M}^{*}_{n,p}$ the set of all type-$*$ maps of the $p$-gon with
$n$ inner vertices.  Note that since quadrangulations are bi-partite,
$\mathcal{M}^\square_{n,p} = \emptyset$ for $p$ odd, hence in the case of
quadrangulations we implicitly restrict all statements to $p$ even.
For $p \geq 1$, let $M^{*}_{n,p}$ be a random variable uniformly
distributed over the sets $\mathcal{M}_{n,p}^{*}$. Then we have the
following convergences in distribution for the distance
$\mathrm{d_{loc}}$:
\[
M_{n,p}^{*} \xrightarrow[n\to\infty]{(d)} M_{\infty,p}^{*}.
\]
 Extending the previous terminology we call these
objects the UIP$*$ of the $p$-gon.

The preceding convergences are easy corollaries of the convergences with no
boundary \eqref{UIPM}, proved by conditioning on the root having a suitable
neighborhood and removing that neighborhood to get the boundary (see
\cite{Ang03,CMboundary} for details).

We can now introduce the main characters of our work: the half-plane
UIP$*$. These are obtained as limit of the UIPT (resp.\,UIPQ) of the
$p$-gon as $p\to\infty$.  More precisely, we have the following convergence
in distribution for $ \mathrm{d_{loc}}$
\begin{equation}
  M_{\infty,p}^{*}  \xrightarrow[p\to\infty]{(d)}  \bM^*.
  \label{1/2UIPM}
\end{equation}
(As noted, in the case $* = \square$ the convergence holds along even
values of $p$.)  The random infinite planar map $\mathbf{M}^*$ is called
the half-plane UIPT (resp.\,UIPQ) which we abbreviate by UIHP$*$.  The
convergence \eqref{1/2UIPM} was established in \cite{Ang05} in the
case of triangulations and can be easily adapted to the quadrangulation case.
See also \cite{CMboundary} for a different construction of
$\mathbf{M}^\square$ via bijective techniques ``\`a la Schaeffer''
\cite{Sch98}.

\paragraph{Percolation.}
Having introduced the random lattices, let us specify the models of
percolation that we will discuss.  Conditionally on $\mathbf{M}^*$, we
consider Bernoulli percolation on the edges, vertices or faces, that is, we
color the elements of the map white with probability $p \in (0,1)$ and
black with probability $1-p$ independently from each other, and consider
the structure of connected white clusters.  If we color the edges, we speak
of bond percolation, if we color the vertices we speak of site percolation.
Coloring faces yields site percolation on the dual of the map (two faces
are adjacent if they share an edge) and will be called ``face percolation''
in this work.

In the triangular case $*\in\{\total\}$, site percolation has already been
analyzed in \cite{Ang05,Ang03} where it is proved that $p_{c,\site}^* =
\frac{1}{2}$. The techniques developed in this paper do not apply to site
percolation on general planar maps (other than triangulations) and for
example the value of $p_{c,\site}^\square$ is still unknown. However in the
case of \emph{bond} or \emph{face} percolation we prove that the critical
percolation thresholds are almost surely constant and can be expressed by a
universal formula relying on a unique parameter depending on the model.  To
give their values, we introduce for each model of planar map a quantity
$\delta^*>0$.  This quantity is well defined and can be computed for fairly
general models of planar maps.  In the main classes we study we have
\begin{eqnarray} \label{delta}
  \delta^\one = \frac{1}{\sqrt{3}}, \qquad \delta^{\two} = \frac{2}{3} \qquad
  \mbox{and} \qquad \delta^\square = 1.
\end{eqnarray}
\begin{theorem}[Percolation thresholds] \label{thm:thresholds}
  For $* \in \{\total, \square\}$, the critical thresholds for bond and face percolations are almost surely constant and are given
  by 
  \[
  p_{c,\bond}^* =  \frac{\delta^*}{2+\delta^*}
  \qquad \mbox{and} \qquad
  p_{c, \mathrm{face}}^* = \frac{\delta^*+2}{2\delta^*+2}.
  \]
\end{theorem}

We prove that in each model considered in \cref{thm:thresholds}, at the
critical probability, there is no infinite cluster. We also study the
associated dual percolations, which in the case of bond percolation is just
bond percolation on the dual lattice (edges in a map are in bijection with
edges in the dual map, so bond percolation on a map and on its dual use the
same randomness, and are dual to each other) and prove the unsurprising
identity
\[
p_{c,\bond}^* = 1 - p_{c,\bond'}^*.
\]
In the case of face percolation, the dual percolation is the same
percolation but where faces are declared adjacent if they share a vertex.
We call it the face$'$ percolation; here also $p_{c,\face}^* = 1 -
p_{c,\face'}^*$.

The universal form of the critical probability thresholds expressed in
\cref{thm:thresholds} in terms of $\delta^*$ holds in a much larger list of
maps than the ones we consider in this work and could be applied, e.g.\ to
pentagulations, general planar maps or planar maps with Boltzmann
distribution, see \cite{MM07}.  The only quantity to compute would be the
equivalent of $\delta^*$ defined in \cref{prop:mean}. Notice also that
$p_{c, \bond}^*$ and $p_{c, \face}^*$ are functions of each other.  If an
oracle such as a physics conjecture or a self-duality property, etc.\
furnishes one of the two thresholds then \cref{thm:thresholds}
automatically gives the other.

In relation to \cite{Ang03}, one key idea which enables us to treat bond
percolation is to keep as much randomness as we can during the exploration
process. In other words, even after being discovered in the map, the status
of an edge can be kept random until it is necessary for the exploration
process to know its color.

\paragraph{Critical exponents.}
In contrast with the critical threshold values which depend on the local
features of model considered, we also compute a few critical exponents
which are not model-depend.  The exploration of percolation interfaces in
random maps involves random walks with heavy-tailed step distribution in
the domain of attraction of a totally asymmetric (spectrally negative)
stable law of parameter $3/2$.  Using standard results for heavy-tailed
random walks we are able to compute critical exponents related to the
perimeter (boundary) and the volume of critical percolations clusters.

For sake of simplicity we restricted our proof to the case of site
percolation on triangular lattices but there is no doubt that our methods
could be adapted to more general cases and would yield the same critical
exponents.  We now make our setting precise.  In $\bM^{\two}$ we consider
the hull $\cH$ of the cluster of a unique white vertex among a full black
boundary. That is, we fill the finite holes in the map created by the
cluster.  We will consider the volume $|\cH|$ of $\cH$ that is its number
of vertices and its boundary $|\partial\cH|$ which is the number of
vertices of $\cH$ adjacent to $\bM^*\setminus\cH$.  We shall also consider
the extended hull $\cH^1$ which is the hull formed by all the triangles
adjacent to $\cH$.  More precisely we are interested in the boundary
$|\partial\cH^1|$ of this extended hull.

\begin{figure}[!h]
  \begin{center}
    \includegraphics[width=8cm]{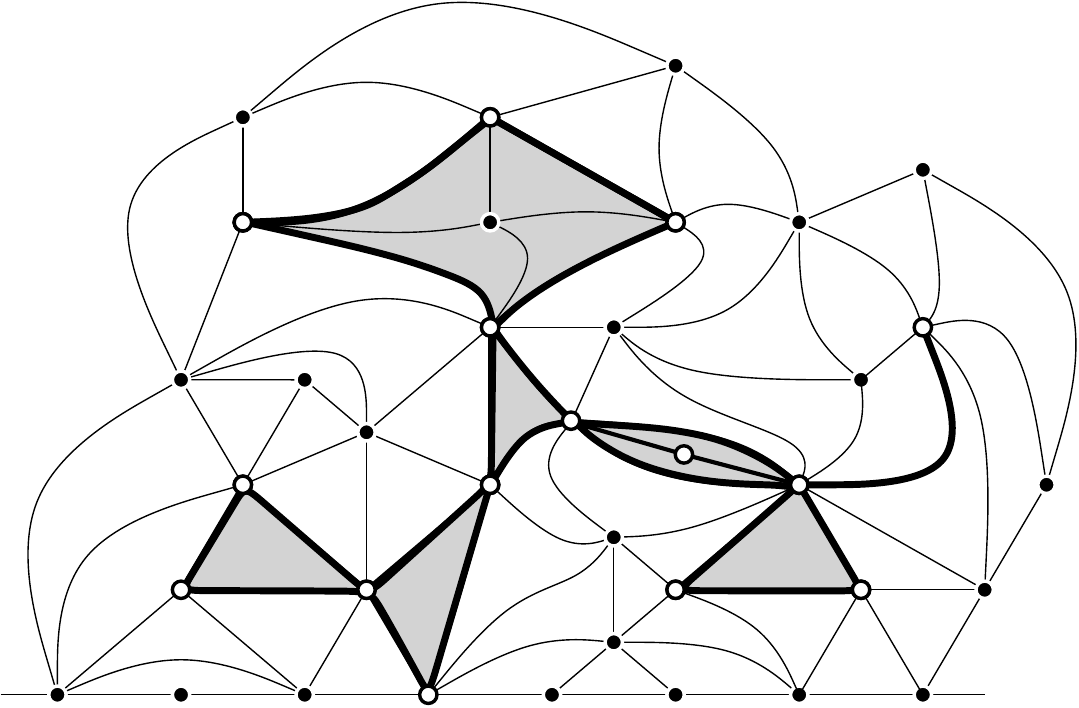}
    \caption{An example of cluster, its extended hull and its hull in
      gray.}
    \label{fig:hulls}
  \end{center}
\end{figure}
 
\begin{theorem}[Critical exponents]\label{thm:criticexpo}
  At the critical percolation threshold $p_{c} = 1/2$, we have the
  following estimates
  \begin{align*}
    &(i)   & \P_{p_{c}}( |\cH| >n)             &=  n^{-1/4 + o(1)}, & &\\
    &(ii)  & \P_{p_{c}}( |\partial \cH| > n)   &\asymp n^{-1/3}, & &\\
    &(iii) & \P_{p_{c}}( |\partial \cH^1| > n) &= n^{-1/2+o(1)}. & &
  \end{align*}
\end{theorem}

Here and later we use the notation $A \asymp B$ to denote that $A/B$ is
bounded below and above by some absolute constants.  The estimates we get
in the proof are significantly more explicit than the above statement.  In
particular, for $(i)$ and $(iii)$ we get lower and upper bounds with only a
poly-logarithmic correction, see \cref{sec:exponents}.  Of course, we
expect these asymptotics to hold with no correction to the polynomial term.

As can be seen in the last theorem the perimeter of the hull and the
perimeter of the extended hull have completely different exponents: after
filling-in the ``fjords'' created by a percolation cluster we drastically
reduces its perimeter. This fact is well-known in the physicist literature
and it well understood for site percolation on the regular triangular
lattice thanks to the SLE processes.

One motivation for this research is the physics theory of $2$-dimensional
quantum gravity.  In particular the KPZ relation \cite{KPZ88} predicts
connections between critical exponents of statistical mechanics models on a
regular lattice and on a random lattice: if a set defined in terms of a
percolation process on a fixed regular lattice has dimension $2(1-x)$ and
and the corresponding set defined in terms of percolation has ``dimension''
$4(1-\Delta)$, then the KPZ relation states that
\[
x = \frac{\Delta(2\Delta+1)}{3}.
\]
For models other than percolation a similar relation holds, with the
coefficients in the quadratic relation given in terms of the so-called central
charge $c$ associated with the model.  While there is some recent progress
towards understanding this relation in the work of Duplantier \& Sheffield
\cite{DS11}, the KPZ relation remains unproved.  Our work gives another
strong indication that the relation does hold for critical percolation
interface and boundary of the hull generated, in other words for SLE$_{6}$
and SLE$_{8/3}$.  This relation
has already been checked for various other sets including e.g.\ pioneer
points of simple random walk \cite{BCsubdiffusive}.

\medskip

The paper is organized as follows. In \cref{sec:peeling} we present in a
unified way some enumeration results on random planar maps and introduce
the spatial Markov property as well as the peeling process which are the
key ingredients of this work. \cref{sec:threshold} is devoted to
identifying the critical threshold parameters presented in
\cref{thm:thresholds}. The last section uses classical results on
heavy-tailed random walk to get off-critical and critical exponents.

\paragraph{Acknowledgments.}
This work was started during overlapping visits of the authors to Microsoft
Research in 2010. We thank our hosts for their hospitality. We are also
grateful to Igor Kortchemski for sharing ideas and references about
discrete stable processes.

\section{Peeling process}
\label{sec:peeling}

\subsection{A few properties of $\bM^*$}

The random infinite half-planar maps $\bM^*$ for $* \in \{\total,
\square\}$ were obtained as local limits of uniform $*$-angulations of the
$p$-gon with $n$ faces by first letting $n \to \infty$ and then sending $p$
to infinity.  Since the distribution of uniform $*$-angulation of the
$p$-gon is invariant under re-rooting along the boundary the same property
holds true for $\bM^*$.  More precisely, $\bM^*$ has an infinite simple
boundary that can be identified with $\Z$, the root edge being
$0\to1$.  Then for every $k \in \Z$, the law of the UIHP$*$ re-rooted at
the edge $k \to k+1$ is the same as the original distribution of $\bM^*$.
Because of this invariance we allow ourselves to be imprecise at times
about the location of the root edge of $\bM^*$.

A less trivial property satisfied by $\bM^*$ is one-endness: it has been
proved in \cite{Ang05} in the triangulation case and in \cite{CMboundary}
for the quadrangulation case that $\bM^*$ almost surely has one end (recall
that a graph is one-ended if the complement of any finite subgraph $A$
contains a unique infinite connected component).  Roughly speaking, there is
a unique way to infinity in $\bM^*$.

However, the foremost property of $\bM^*$ is the spatial Markov property.
The half-planar model has the most simple form of spatial Markov property
which, roughly speaking, states that the complement of a simply connected
region of $\bM^*$ that contains the root edge (properly explored) is
independent of this region and is distributed according to $\bM^*$.  This
property, also called the domain Markov property in this context, is
explored further in \cite{AR13}.  In order to make this statement precise we
shall need some enumerative background.

\subsection{Enumeration}

We gather here several results about enumeration and asymptotic enumeration
of planar maps.  Recall that for $* = \total,\square$ respectively, and for $n \geq 0$, $p \geq 1$ we denote by $\cM^*_{n,p}$ the sets of all
type $1$ or $2$ triangulations and the set of quadrangulations of the
$p$-gon with $n$ inner vertices.  The reader should keep in mind that
$\cM^\square_{n,p} = \varnothing$ if $p$ is odd.  By convention the set
$\cM_{0,2}^*$ contains the unique map (with simple boundary) composed of a
single oriented edge.

All the results presented here can easily be deduced from the exact
formulae for $\# \cM_{n,p}^*$ (or the intermediate steps to reach them) and
can be found in \cite{GJ83} for $*=\two$, in \cite{Kri07} for $*=\one$
and in \cite{BG09} for $*=\square$.  By convention the asymptotics for
$\#\cM^\square_{n,p}$  only apply to even values of $p$.

\bigskip

For $n\geq 0$ and $p \geq1$ we have the following asymptotics for
$\# \cM_{n,p}^*$:
\begin{equation}
  \# \cM_{n,p}^{*} \underset{n\to\infty}{\sim}  C_*(p) \rho_{*}^n n^{-5/2}, 
  \label{equivalentn} 
\end{equation}
where
\[
\rho_{\one} = \sqrt{432},\qquad  \rho_ \two=27/2  \qquad \mbox{and} \qquad
\rho_\square = 12.
\]

The asymptotics \eqref{equivalentn} in general and the exponent $5/2$ are
typical to enumeration of planar maps and hold for many other classes of
planar maps. As in previous works and as we will see below, the exponent
$5/2$ plays a crucial role in the large scale structure of the random lattices.
Furthermore, the functions $C_*$ also have a universal asymptotic behavior:
\begin{equation}
  C_{*}(p) \underset{p\to\infty}{\sim}  K_* \alpha_*^p \sqrt{p},
  \label{equivalentp}
\end{equation}
where
\[
\alpha_{\one} = 12, \qquad \alpha_\two=9, \qquad \mbox{and} \qquad
\alpha_\square=\sqrt{54}.
\]
The exact values $K_{\triangle_{1}} = (36 \sqrt{2} \pi)^{-1},
K_{\triangle_{2}}=(54 \pi \sqrt{3})^{-1}$ and finally $K_\square=(8
\sqrt{3} \pi)^{-1}$  will not be relevant in what follows
but we furnish them for completness.  Thanks to the $n^{-5/2}$ polynomial
correction in the asymptotic \eqref{equivalentn} the series $\sum_{n\geq 0}
\# \cM^*_{n,p}\rho_*^{-n}$ converges and we denote its sum by $Z_*(p) <
\infty$.  In fact, for $*\in\allthree$, the functions $Z_*(p)$ can be
exactly computed and all exhibit an asymptotic behaviour of the form
$Z_*(p) \sim \kappa_* p^{-5/2} \alpha_*^{p}$ with $\kappa_*>0$, more
precisely we have
\begin{align*}
  Z_{\one}(p) &= \frac{(2p-5)!! 6^p}{8 \sqrt{3} p!} && \text{for } p\geq 2
  \quad \text{and}\quad  Z_{\one}(1) =  \frac{2 - \sqrt{3}}{4},\\ 
  Z_{\two}(p) &= \frac{(2p-4)!}{(p-2)!p!}\left( \frac{9}{4}\right)^{p-1}
  &&\text{for } p \geq 2,\\ 
  Z_{\square}(2p) &= \frac{8^p (3p-4)!}{(p-2)! (2p)!} 
  &&\text{for }p \geq 1. 
\end{align*}
(We use the notation $(2n+1)!!= (2n+1) (2n-1)... 3 \cdot1$ and $(-1)!!=1$.)
The reader may identify $Z$ as the partition function in the following
measure:

\begin{definition}
  The free $*$-Boltzmann distribution of the $p$-gon is the probability
  measure on $\bigcup_{n\geq 0}\cM_{n,p}^*$ that assigns a weight
  $\rho^{-n}_* Z_*(p)^{-1}$ to each map belonging to $\cM^*_{n,p}$.
\end{definition}

\subsection{The spatial Markov property}

\subsubsection{One-step peeling of $\bM^*$}
\label{one-step}

We now present the version of the spatial Markov property (also called the
domain Markov property \cite{AR13}) that we use.  This version describes
the conditional laws of the different sub-maps we obtain from $\bM^*$ after
conditioning on the face that contains the root edge.  We do not present
the proofs since they are contained in \cite{Ang05} for the case of
triangulations ($* = \one,\two$) and can easily be adapted to the case of
quadrangulations.  We do however include a rough sketch of the calculations
involved.
 
Let $\bM^*$ be a uniform infinite planar map of the half-plane.  Assume
that we reveal in $\bM^*$ the face on the left of the root edge, we call
this operation {\bf peeling at the root edge}. The revealed face can
separate the map into many regions and different situations may appear
depending on the type of planar map we consider.  Let us make a list of the
possibilities and describe the probabilities and the conditional laws for
each case.
 
\paragraph{Triangulation case.}

In this paragraph $* \in \{\total\}$.  We reveal the triangle that
contains the root edge in $\bM^*$.  Two cases may occur:

\begin{itemize} 
\item The revealed triangle could simply be a triangle with a third vertex
  lying in the interior of $\bM^*$, see \cref{fig:peelhalftrig1}(a).  This
  event appears with probability which we denote by $q_{-1}^*$, and it is easy
  to see from the convergences \eqref{UIPM} and \eqref{1/2UIPM} the
  asymptotics \eqref{equivalentn} and \eqref{equivalentp} that
  \[
    q^*_{-1} =  \lim_{p\to\infty} \lim_{n\to\infty}
    \frac{\# \cM^*_{n-1,p+1}}{\# \cM^*_{n,p}}
    = \frac{\alpha_*}{\rho_*}.
  \]
  We deduce that $q^\two_{-1} = 2/3$ and $q^\one_{-1} = 1/\sqrt{3}$.

  Furthermore, conditionally on this event, the remaining triangulation (in
  light gray in \cref{fig:peelhalftrig1}) has the same distribution as $\bM^*$.
  To be precise, we need to specify a root for this new map, but due to the
  translation invariance discussed above, any boundary edge will do. For
  example we may root it at the edge of the revealed triangle which is
  adjacent on the left of the original root edge. 

  \begin{figure}[h]
    \begin{center}
      \includegraphics[height=2.5cm]{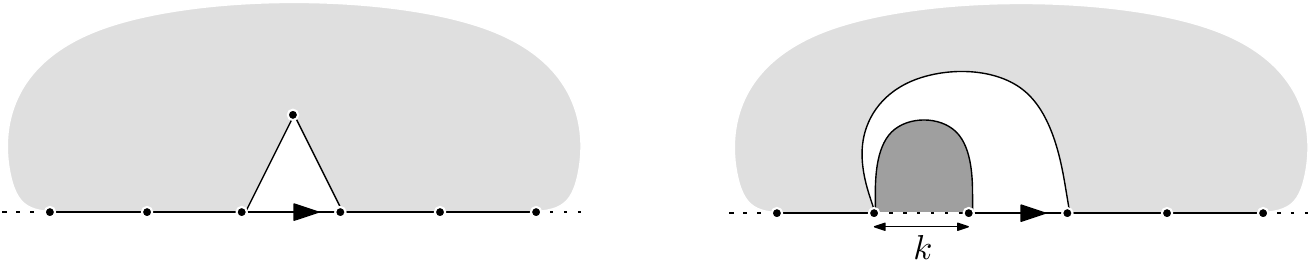}
      \caption{Cases when peeling a triangulation.  The map in the
        light gray area has the same law as the entire map.}
      \label{fig:peelhalftrig1}
    \end{center}
  \end{figure}

\item Otherwise, the revealed triangle has all of its three vertices lying
  on the boundary and the third one is either $k \geq 0$ edges to the left
  of the root edge or $k$ edges to the right of the root edge, see
  \cref{fig:peelhalftrig1}(b).  These two events have the same
  probability which we denote by $q^*_k$.  Notice first that when $*=\two$ we
  must have $k>0$ since loops are not allowed.  Here also, one can use
  \eqref{UIPM} and \eqref{1/2UIPM} to compute $q_k^*$, and we get
  
  \begin{align*}
    q^*_k &= \lim_{p \to \infty} \lim_{n \to \infty}
    \sum_{n_{1}+n_{2}=n}\frac{ \# \cM^*_{n_{1},k+1} \# \cM^*_{n_{2},p-k}}{ \#
      \cM^*_{n,p}} \\
    &= \lim_{p\to\infty} Z_{*}(k+1) \frac{C_{*}(p-k)}{C_{*}(p)}
     + Z_{*}(p-k) \frac{C_{*}(k+1)}{C_{*}(p)} \\
    &= Z_{*}(k+1) \alpha_*^{-k}.
  \end{align*}
  
  Furthermore, conditionally on the fact that the revealed triangle has its
  third vertex lying $k$ edges away from the root edge, the triangulation
  with finite simple boundary it encloses (in dark gray on
  \cref{fig:peelhalftrig1}(b)) is distributed according to a $*$-Boltzmann
  of the $k+1$-gon.  The remaining infinite part (in light gray on the
  figure) with arbitrary choice of root, is independent of the finite map
  enclosed and is distributed according to $\bM^*$.  The $k$ edges
  separating the root edge from the third vertex are called the {\bf
    swallowed boundary}.

\end{itemize}

\paragraph{Quadrangulation case.}
Let $\bM^\square$ be a half-plane UIPQ and let us reveal the quadrangle that
contains the root edge.  We have three different cases.

\begin{itemize}
\item The simplest of all is the case when the quadrangle containing the
  root edge has two of its vertices lying inside $\bM^\square$.  As for
  triangulations, we may compute the probability of this event to be
  \[
  q^\square_{-1} 
  = \lim_{p\to\infty} \lim_{n\to\infty} \frac{\# \cM^\square_{n-2,p+2}}{\#
    \cM^\square_{n,p}}
  = \left(\frac{\alpha_{\square}}{\rho_{\square}}\right)^2 = \frac{3}{8}.
  \]
  Here also, conditionally on this event, the remaining half-plane
  quadrangulation (rooted arbitrarily at the first edge of the revealed
  quadrangle) is distributed according to $\bM^\square$.
  \begin{center}
    \includegraphics[width=45mm]{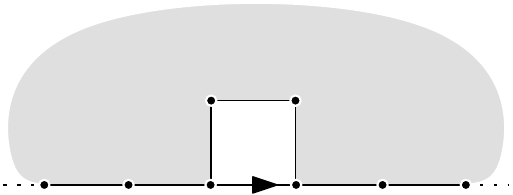}
  \end{center}
  
\item The revealed square could also have three vertices lying on the
  boundary of the map and one in the interior and separate the map into a
  region with a finite boundary and one with an infinite boundary.  This
  again separates into two sub-cases depending whether the third vertex is
  lying on the left or on the right of the root edge, by symmetry these
  events have the same probability.  Suppose for example that the vertex is
  on the left of the root edge.  This further splits according to whether
  the fourth vertex of the quadrangle lies on the boundary of the finite
  region or of the infinite region.  Since all quadrangulations are
  bipartite, this is determined by the parity of the number of edges
  between this third point on the boundary and the root edge, which we
  denote by $k$ when odd and by $k'$ when even (see figure below).

  \begin{center}
    \includegraphics[width=11cm]{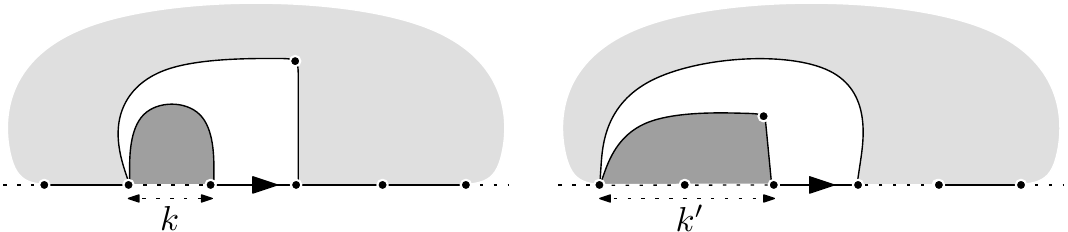}
  \end{center}
  
  If $k$ -- the length of the swallowed boundary -- is odd then the fourth
  point of the discovered square lies on the boundary which is exposed to
  infinity.  This event has a probability
  \[
  q^\square_k
  = \lim_{p\to\infty} \lim_{n\to\infty} \sum_{n_{1}+n_{2}=n-1}
  \frac{ \# \cM^\square_{n_{1},k+1} \# \cM^\square_{n_{2},p-k+1}}{ \#
    \cM^\square_{n,p}}
  = \frac{Z_{\square}(k+1)\alpha_{\square}^{1-k}}{\rho_{\square}}.
  \]

  On the other hand, if $k'$ -- the length of the swallowed boundary -- is
  even then the fourth point of the square must lie in the enclosed region
  and this event has a probability
  \[
  q^\square_{k'} = \lim_{p \to \infty} \lim_{n \to \infty}
  \sum_{n_{1}+n_{2}=n-1}\frac{ \# \cM^\square_{n_{1},k'+2} \#
    \cM^\square_{n_{2},p-k'}}{ \# \cM^\square_{n,p}}
  = \frac{Z_{\square}(k'+2) \alpha_{\square}^{-k'}}{\rho_{\square}}.
  \]

  In both cases, conditionally on any of these events the enclosed maps are
  $\square$-Boltzmann of the $\ell$-gon where $\ell = k+1$ or $\ell=k'+2$
  and the infinite remaining part is independent of it and has the same
  distribution as $\bM^\square$.
 
\item The last case to consider is when the revealed square has all of its
  four vertices on the boundary.  This could happen in three ways, as $0$,
  $1$, or $2$ vertices could be to the right of the root edge (see
  \cref{fig:peelhalfquad3}).  In this case the revealed quadrangle
  separates from infinity two segments along the boundary of lengths $k_1$
  and $k_2$ as depicted on the figure below.  The numbers $k_1$ and $k_2$
  must both be odd.  These events have the same probability
  \begin{align*}
    q^\square_{k_1,k_2} &=
    \lim_{p\to\infty} \lim_{n\to\infty}
    \sum_{n_1+n_2+n_3 = n} \frac{ \# \cM^\square_{n_1,k_1+1}
      \#\cM^\square_{n_2,k_2+1} \#\cM^\square_{n_{3},p-k_1-k_2}}
    {\#\cM^\square_{n,p}} \\
    &= Z_{\square}(k_1+1) Z_{\square}(k_2+1) \alpha_{\square}^{-k_1-k_2}.
  \end{align*}
  As in all other cases, conditionally on any of these events the three
  components are independent, the finite ones are $\square$-Boltzmann of
  proper perimeters and the infinite one is distributed as $\bM^\square$. 

  \begin{figure}[h]
    \begin{center}
      \includegraphics[width=12cm]{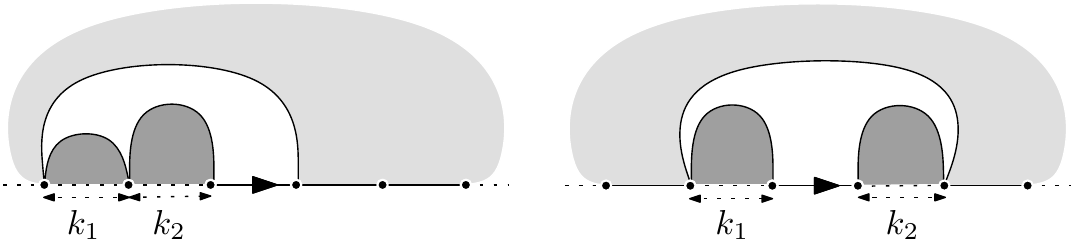}
      \caption{ Two of the ways a revealed quadrangle may have all its
        vertices on the boundary.}
      \label{fig:peelhalfquad3}
    \end{center}
  \end{figure}
\end{itemize}

\paragraph{General case.}
This method applies to more general maps, including $d$-angulations for any
$d$ (odd or even) as well as maps with mixed face sizes and a weight for
each face size.  In complete generality the analogues of
\eqref{equivalentn} and \eqref{equivalentp} are not known, though they are
believed to hold, and are known in some cases, most notably fairly general
bipartite maps \cite{BDFG04}.  In any class of maps where these asymptotics
hold, a similar peeling procedure may be applied. The number of cases grows
exponentially in $d$, as each vertex of the revealed face may or may not be
on the boundary, and in general some of the vertices may coincide.
However, the separated components of the map are always independent
Boltzmann maps, and are independent of the remaining infinite part which is
distributed as the half-plane model.  While the computational complexity of
such analysis increases quickly with $d$, it seems there is no conceptual
difficulty involved in generalizing our arguments to any specific $d$.

\subsubsection{Starring $\delta^*$}

Although the one-step peeling transitions in the cases of triangulations
and quadrangulations seem different they share several common key
properties which specify here.  To this end, let us introduce a few
notions.  Imagine that we reveal the face adjacent to the root edge in
$\bM^*$ as above.  The new face may enclose a finite region (or two) and
can surround some of the edges of $\partial \bM^*$. We call these edges the
{\bf swallowed edges}.

On the other hand, some edges of the new discovered face form a part of the
boundary of the remaining half-planar map.  These edges are called {\bf
  exposed edges}.  In the triangulation case there are two exposed edges
when the discovered triangle has only two vertices lying on the boundary
(first case) and one exposed edge otherwise.  In the quadrangulation there
are three exposed edges on the event of probability $q_{-1}^\square$, two exposed edges on the events of probabilities
$q_{k}^\square$ for $k\geq 1$ odd and only
one on the events of probabilities $q_{k'}^\square$ for $k'\geq 2$ even and
$q_{k_1,k_2}^\square$.  See \cref{exposed-swallowed}.

\begin{figure}[!h]
  \begin{center}
    \includegraphics[width=\textwidth]{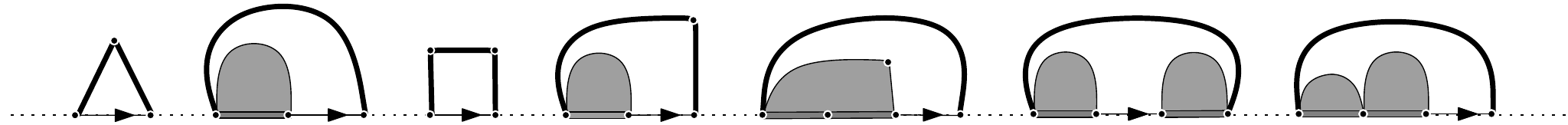}
    \caption{The exposed edges are in fat black lines and the swallowed
      ones are in fat gray lines.}
    \label{exposed-swallowed}
  \end{center}
\end{figure}

Let $\cE^*, \cR^*$ respectively be the number of exposed edges and the
number of edges swallowed \emph{to the right} of the peeling point when
revealing a single face. By symmetry, the number of edges swallowed to the
left of the peeling location has the same distribution as $\cR$. Of
course, the number of edges swallowed on the two sides, and $\cE$ are not
independent.  We now define 
\begin{equation}
  \label{eq:deltadef}
  \delta^* := \E[\#\mathrm{Swallowed\ edges}] = 2\E[\cR].
\end{equation}
We will see that $\delta^*$ plays a key role in determining percolation
thresholds on these infinite maps.

\begin{proposition}\label{prop:mean}
  We have
  \begin{align}\label{mean}
    \E [\cE^*] &= 1+\delta^* & \mbox{and}&&\E[\cR^*] &= \frac{\delta^*}{2}.
  \end{align}
  Moreover, for $*\in\allthree$ we have
  \begin{align}
    \delta^\one &= \frac1{\sqrt{3}}, & \delta^{\two} &= \frac{2}{3}, \qquad
    \text{and} & \delta^\square &= 1.
  \end{align}
\end{proposition}

\begin{proof}
  Both statements follow from a direct computation using the exact
  expression of the probabilities $q^*_.$ and the enumerative formulae of
  the last section.  We omit the details, though the result is easily and
  reliably verified in a computer algebra system such as
  Mathematica\texttrademark\ or Maple\texttrademark.
\end{proof}

\begin{remark}
  The identity \eqref{mean} should hold for any reasonable class of planar
  maps.  It would be nice to have a conceptual explanation for it, rather
  than a computational proof, perhaps in terms of singularities of
  generating functions.
\end{remark}

Note that in the triangulation case this implies that $\delta^*$ is simply
equal to $q_{-1}^*$, since \mbox{$\cE\in\{1,2\}$}.  The relation between
the expected number of swallowed and exposed edges can also be interpreted
as follows.  During the exploration of a face the change in the length of
the boundary of the external infinite half-plane map has zero
expectation. Indeed the initial edge at which we peel, together with the
swallowed edges are no longer on the boundary, and the exposed edges are
added.

Note however that the number of exposed edges is always bounded by $2$ in
the triangulation case and by $3$ in the quadrangulation case, whereas the
number of swallowed edges has a heavy tail.  Indeed, $\cR^*$ has a
heavy-tail of index $5/2$ that is
\[
\P(\cR^* = k) \sim  c^* k^{-5/2} \qquad \text{as } k\to\infty.
\]
In particular, $\cR^*$ is in the domain of attraction of a spectrally
negative $ \frac{3}{2}$-stable random variable, a fact we use in
\cref{sec:exponents}.

\subsubsection{Markovian exploration: the peeling process}

Based on the description of the one-step peeling of the root edge one can
define a growth algorithm for random maps, {\bf the peeling process}, that
was first used heuristically by physicists (see \cite{Wat95} and
\cite[Section 4.7]{ADJ97}) in the theory of dynamical triangulations.
Angel \cite{Ang05,Ang03} then defined it rigorously and used it to study
the volume growth and site percolation on the uniform infinite planar
triangulation $M_\infty^*$ for $* \in \{\total\}$.  See also
\cite{BCsubdiffusive} where the peeling has been used to study the simple
random walk on the UIPQ.  We adapt these ideas to the context of half-plane
UIP$*$.

\medskip

Let $M$ be an infinite $*$-angulation with an infinite simple boundary. If
$a$ is an edge on the boundary of $M$ we denote the {\bf one-step peeling
  outcome} by $\mathrm{Peel}(M,a)$.  This is the map obtained from $M$ by
``removing'' the submap made of the face adjacent to $a$ together with any
finite regions this face encloses, see \cref{one-step}.  This map is rooted
as in the previous section.

A {\bf peeling process} is a randomized algorithm that consists of
exploring $\bM^*$ by revealing at each step one face, together with any
finite regions that it encloses.  More precisely, it can be defined as a
sequence of infinite $*$-angulations with infinite boundary $\dots \subset
\bM_1^* \subset \bM_0^* = \bM^*$ such that for all $i>0$
\[
\bM_i^* = \mathrm{Peel}(\bM^*_{i-1}, a_i)
\]
for a (necessarily unique) edge $a_i$ on the boundary of $\bM_{i-1}^*$.  We
denote the revealed part by $P^*_i$.  This consists of all faces of $\bM^*$
not in $\bM^*_i$, and all vertices and edges contained in them.  Moreover
the choice of the edge $a_i$ should be independent of the unrevealed part
$\bM_{i-1}^*$.  That is $a_i$ can be chosen by looking at the revealed part
$P^*_{i-1}$ made of the union of all the faces revealed and the finite
regions they enclose up to step $i-1$ and possibly an independent source of
randomness which is independent of $\bM_{i-1}^*$.  Note that many different
algorithms can be used in order to choose the next edge to reveal.  The
only constraint is that we do not used information from the undiscovered
part.  Under these hypotheses we have

\begin{proposition}\label{prop:vraipeeling}
  Let $\dots\subset \bM_1^* \subset \bM_0^* = \bM^*$ be a peeling process
  then
  \begin{enumerate}
  \item for every $i \geq 0$, $\bM^*_i$ is distributed as $\bM^*$ and is
    independent of $P_i^*$,
  \item the sequence of pairs $(\cE^*_i,\cR^*_i)_{i\geq 1}$ representing
    the number of exposed edges and the number of edges swallowed to the
    right of the peeling edge $a_i$ for $i \geq 1$ is an i.i.d.\ sequence
    with mean given by \cref{prop:mean},
  \item for $* \in \{\total\}$ these have distribution
    \[
    \P \big( ( \cE_i^*, \cR_i^*) = (e,r) \big) = \begin{cases}
      q^*_{-1} & (e,r) = (2,0),\\ 
      q^*_0 + (1-q^*_{-1})/2 & (e,r) = (1,0) \\
      q^*_k    & (e,r) = (1,k), k>0.
    \end{cases}
    \]
  \end{enumerate}
\end{proposition}

The explicit distribution of $(\cE^\square, \cR^\square)$ can be computed
from the description and formulae in \cref{one-step}, but will not be used
in the following.

\begin{proof}
  We prove the first statement by induction.  Suppose that at step $i+1\geq
  1$ the as yet unrevealed part $\bM_i^*$ is independent of the revealed
  part $P_i^*$ and is distributed as a standard UIHP$*$.  We then pick an
  edge $a_{i+1}$ on the boundary of $\bM_i^*$. Since the choice of this
  edge is independent of $\bM_i^*$ itself, the map $\tilde{\bM}_i^*$
  obtained by re-rooting $\bM_i^*$ at $a_{i+1}$ is also distributed as the
  UIHP$*$.  We can thus reveal the face in $\tilde{\bM}_i^*$ adjacent to
  this edge and deduce from the previous section that $\bM_{i+1}^* =
  \mathrm{Peel}(\tilde{\bM}_i^*, a_{i+1}) = \mathrm{Peel}(\bM_i^*,a_{i+1})$
  is independent of the union of $P_i^*$ and of the finite regions
  discovered by this operation.

  The second point easily follows from these considerations, and the third
  from the description of $q^*_k$ above.  The additional term in the case
  $k=0$ comes from the event that the revealed triangle has its third
  vertex to the left of the root edge, which has probability
  $\frac12(1-q^*_{-1})$.
\end{proof}

\section{Percolation thresholds} \label{sec:threshold}

We now use the peeling exploration described in the last section in order
to study percolation on UIHP$*$. The key idea being as in \cite{Ang03} to
explore the (leftmost) percolation interface. But this needs some care and
tricks depending on the type of percolation and lattice considered.\medskip

To help the reader getting used to the tools and methods, we start by
recalling the exploration of percolation interfaces in site percolation on
triangulations as developed in \cite{Ang05,Ang03}.  We then generalize this
exploration process to treat the case of face and bond percolations on
UIHP$*$.  As we will see, the exploration of site-percolation interfaces is
possible in the triangular lattice, but present methods fail for more
general lattices.  On the other hand our exploration of face and bond
percolations can be performed in virtually any class of map.\medskip

We begin with \cref{thm:site,thm:sitedual,thm:bond,thm:bonddual}, where we
study the cluster of the origin with special boundary conditions. The
structure of the proof of each of these theorems is as follows: First we
introduce a special boundary condition and a peeling algorithm. We then
check that the process leaves the form of the boundary condition invariant
and check that the exploration is Markovian in the sense of the previous
section.  For each model, a planar topological argument shows that the
peeling stops if and only if the cluster of the origin is finite.  We
finally relate the length of the ``active boundary'' during this
exploration to a random walk whose increments have a computable mean
expressed in terms of $p \in (0,1)$ and $ \delta^*$ only. The peeling
threshold is given when these increments have zero
mean. \cref{prop:quenched} then proves that the threshold probabilities
found in these results indeed correspond to the quenched critical
probabilities for percolation in the standard models.

\subsection{Site percolation on  triangulations}  \label{sec:siteperco}

Let $\bM^*$ for $* \in \{\total\}$ be a UIHPT.  Suppose that conditionally
on $\bM^*$ we color the vertices of $\bM^*$ with two colors (black and
white) as follows: We first color deterministically all the vertices of the
boundary in black except the extremity of the root edge that we color in
white.  We then color all the remaining vertices independently in white
with probability $p \in (0,1)$ and in black with probability $1-p$. This
yields site percolation on $\bM^*$ with mostly black boundary condition,
except for a single vertex.  Our goal is to study the white cluster $\cC =
\cC^*_\site$ containing the only white vertex of the boundary.  We shall
use $\cC$ when there is no risk of confusion, and make explicit the type of
lattice and percolation only when needed.

\begin{theorem}[\cite{Ang05,Ang03}]\label{thm:site}
  We have $\P(|\cC^*_\site| = \infty) > 0$ if and only if $p>p^*_{c,\site}$
  where
  \[
  p^{\one}_{c, \site} =  p^{\two}_{c, \site}  =  1/2.
  \]
\end{theorem}

\begin{proof}
  The central idea is to explore along the (leftmost) percolation
  interface using a peeling procedure.  Let us describe precisely the
  algorithm to choose the edges to peel:
  \begin{quote}
    \textsc{Algorithm:} Assume that $\bM_i^*$ is a site-percolated UIHP$*$
    with a boundary condition of the form $\dots - \bullet - \circ - \dots
    - \circ - \bullet -\dots$ (i.e.\ all white vertices form a single
    finite connected segment).  Peel the edge $a_{i+1}:= \bullet - \circ$
    (which is well defined on the assumption).  If the peeling of $a_{i+1}$
    discovers a new vertex inside $\bM_i^*$ (case $1$ in \cref{one-step})
    then also reveal the color of this new vertex.
  \end{quote}
  
  There are a couple of easy facts to check to see that this indeed defines
  a peeling process.  First, it is easy to see that the form of the
  boundary condition black--white--black is preserved after peeling at the
  edge $a_{i}$ and possibly revealing the color of a new vertex. \emph{Here
    we use the fact that $\bM^*$ is a triangulation.}  Second, the
  edge $a_i$ chosen to be peeled at time $i \geq 1$ indeed depends on the
  submap $P_{i-1}^*$ discovered up to time $i-1$ as well as on the color of
  its vertices but clearly does not depend on $\bM_{i-1}^*$, nor on the
  color of its internal vertices.  Hence \cref{prop:vraipeeling} applies,
  at least as long as the boundary of $\bM_i^*$ contains a white region
  necessary to designate the next point to peel.

  \begin{figure}[!h]
    \begin{center}
      \includegraphics[width=\textwidth]{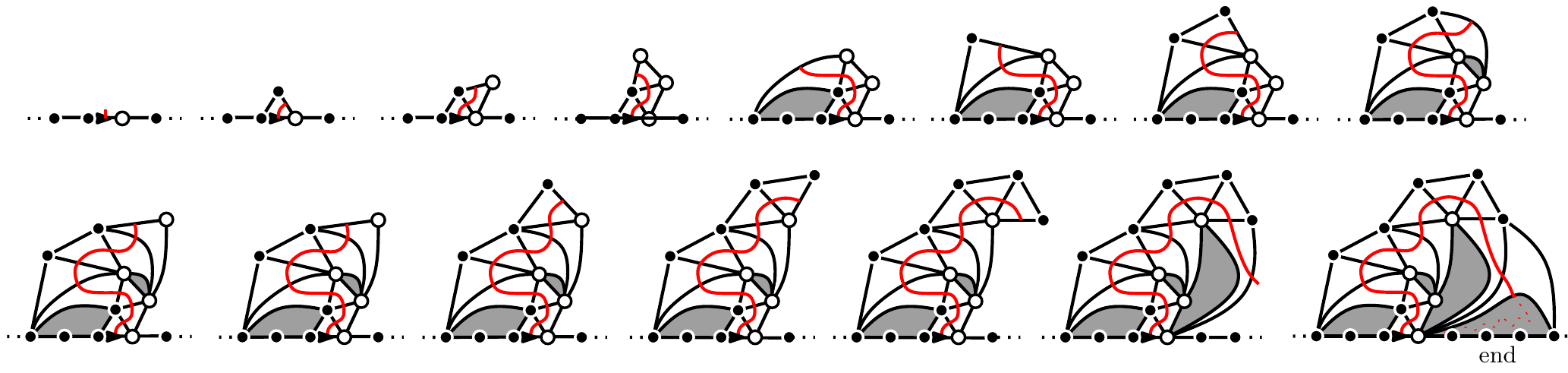}
      \caption{Exploration of the percolation interface in site percolation on a
        UIHPT. The gray part are the finite regions discovered during the
        peeling process. The interface is in red.}
      \label{fig:explorationsite}
    \end{center}
  \end{figure}

  It is easy to check that if this peeling algorithm is used from the very
  beginning then all the white vertices on the boundary of $\bM_i^*$ are
  part of the cluster of the white origin vertex.  This exploration process
  terminates at the first peeling step when the white boundary is
  ``swallowed'', that is, when the new discovered triangle makes a jump to
  the right of the peeling point and reaches the black boundary.  If that
  happens, a simple topological argument shows that the white cluster $\cC$
  must be finite (see \cref{fig:explorationsite}).  If the process does not
  terminate then $\cC$ is infinite.

  Let $S_i$ for $i\geq 0$ be the number of white vertices on the boundary
  of $\bM^*_i$. Thus $S_0=1$.  If $|\cC|=\infty$ then $S_i$ is defined and
  positive for all $i \geq 0$.  On the other hand, if $\cC$ is finite then
  $S_n=0$ for some $n$, after which the above peeling process is no longer
  defined.  (For completeness, we let $S_i=0$ for all $i>n$ in that case.)

  We let $\epsilon_i=1$ if the peeling of $a_i$ discovers a new vertex
  inside $\bM_{i-1}^*$ and if the color of this vertex is white, set
  $\epsilon_i=0$ otherwise.  Notice that conditionally on the fact that the
  face adjacent to $a_i$ in $\bM_{i-1}^*$ has a vertex lying inside
  $\bM_{i-1}^*$ (that is we discover two exposed edges) then $\epsilon_i$
  is a Bernoulli variable of parameter $p$, and is independent of
  $P_{i-1}^*$ and of its coloring.

  \begin{figure}[!h]
    \begin{center}
      \includegraphics[width=16cm]{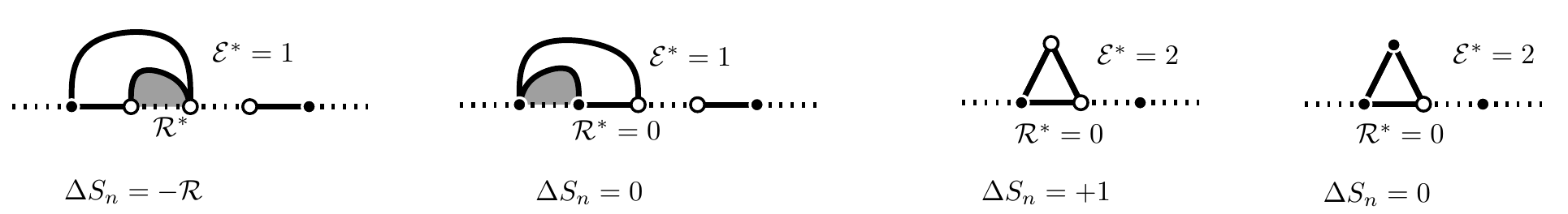}
      \caption{ The different one-step peeling transitions during the
        exploration of site percolation interface with $* \in
        \{\one,\two\}$.}
      \label{fig:casestrig}
    \end{center}
  \end{figure}

  Recall that $(\cE_i^*, \cR_i^*)$ are the number of exposed and swallowed
  edges during the $i$th step of peeling, and that as long as the white
  interface is not empty these are i.i.d.\ random variables whose
  distribution is given in \cref{prop:vraipeeling}.  Then we have the
  following relation between the sequences defined so far that holds for
  $*\in\{\total\}$:
  \begin{equation}\label{eq:explositeperco}
    S_{n} = \Big(S_{n-1} + \epsilon_n 1_{\{\cE^*_n=2\}} - \cR^*_n \Big)^+,
  \end{equation}
  as long as $S_{n-1}>0$ (where $X^+=X\vee0$), see \cref{fig:casestrig}.

  Hence the process $(S_n)$ is a random walk with i.i.d.\ steps starting
  from $1$ and killed (and set to $0$) at the first hitting time of
  $\Z^-=\{0,-1,-2,\dots\}$.  Furthermore the increments of this walk have
  mean
  \[
  \E [ \epsilon(\cE^*-1) - \cR^* ] = \delta^* (p-\frac{1}{2}).
  \]
  It follows that the cluster $\cC$ is almost surely finite if and only if
  $\E [\Delta S] \leq 0$, thus $p_{c,\site}^* = 1/2$ for $* \in \{\total\}$
  and that the cluster of the origin is finite at $p=p_{c,\site}^*$.
\end{proof}

{\sc Interface.}  It is easy to see that the above peeling process just
explores the leftmost interface of the cluster of the origin.  More
precisely, there is a well defined path in the dual graph separating $\cC$
from the black cluster containing the left part of the boundary.  At each
step of the peeling process we reveal a face along this interface.  We also
discover the color of the vertices in the same time we peel $\bM^*$.  Faces
visited by the interface correspond to peeling steps except that part of
the interface that is contained in the triangulation enclosed by the last
jump (in dotted red line on \cref{fig:explorationsite}).

\begin{remark}
  Note that it is essential that $\bM^*$ is a triangulation for the
  exploration of the interface of site percolation to work.  Indeed the
  boundary condition black--white--black may not be conserved during an
  exploration of site percolation on quadrangulations or other maps.  We do
  not know how to adapt these ideas for quadrangulations, and the value of
  $p^\square_{c,\site}$ is unknown. Note also that the fact that
  $p_{c,\site} =1/2$ for triangulations is not surprising since site
  percolation is self-dual on any triangulation.
\end{remark}

\subsection{Face percolation} \label{sec:dual_site}

We now use similar arguments to study face percolation on UIHP$*$ (two
faces are adjacent if they share an edge). Equivalently this is site
percolation on their dual lattices. This time we are not restricted to the
triangulation case. Our results are valid for $* \in \allthree$, and there
is no serious obstacle to deriving them for more general maps.

Let each face of $\bM^*$ be colored independently white with probability
$p\in (0,1)$ and black otherwise.  Face percolation does not have any
apparent boundary condition, but in a sense we explain now, it does.  Let
us color the infinite external face in black, which corresponds to an all black
boundary condition.  As far as percolation clusters are concerned, this is
equivalent to adding an extra face adjacent to each boundary edge and
colouring it black.  We do this for all boundary edges {\em except for the
  root edge}.  For the root edge, we add a white external face
see \cref{fig:explorationdual}.  We now consider
$\cC=\cC^*_{\face}$ to be the white cluster of this ``origin face'', and
show that it is a.s.\ finite if and only if $p \leq p_{c,\face}^*$.

\begin{theorem}\label{thm:sitedual}
  We have $\P(|\cC^*_{\face}| = \infty) >0$ if and only if $p >
  p^*_{c,\face}$ where
  \[
  p_{c,\face}^* = \frac{\delta^*+2}{2\delta^*+2}.
  \]
\end{theorem}

\begin{proof}
  We adapt the exploration process of the last section.  Having added the
  starting white face outside $\bM^*$ and after coloring the infinite
  remaining face in black we have a black--white--black boundary condition
  similar to the situation with site-percolation.  Each edge of $\bM^*_i$
  is incident to one face inside and one outside of $\bM^*_i$.  The form of
  the boundary conditions that we maintain is as follows:

  \begin{quote}
    \textsc{Algorithm:} Assume that the boundary of $\bM_i^*$ is of the
    following form: there is a single connected finite segment along the
    boundary adjacent to white external faces and all other edges are
    adjacent to black faces.  That is we have a black--white--black
    boundary condition.  We then peel the leftmost edge $a_{i+1}$ of the
    white part and reveal the color of the discovered face.
  \end{quote}

  Here also a few checkings are in order.  The initial boundary condition
  is trivially of the required form.  Second, notice that the boundary
  condition black--white--black is preserved by the peeling of the leftmost
  ``white'' edge, and this is valid for $* \in \allthree$ (and indeed
  for any planar map).  Next, as before, the choice of the edge to peel
  stays independent of the unknown region and of its coloring.  
  
  If this algorithm is used from the beginning then every edge on the
  boundary of $\bM_i^*$ that is adjacent to a white face of $P_i^*$ is
  actually connected is the dual of the map $P_i^*$ to the origin white
  face.  Here also the exploration stops when there is no edge to peel and
  a simple topological argument shows that this happens after a finite
  number of steps precisely when the white cluster of the origin is
  surrounded by black faces and is consequently finite.  Otherwise the
  peeling process goes on forever (see \cref{fig:explorationdual}).
  
  \begin{figure}[!h]
    \includegraphics[width=\textwidth]{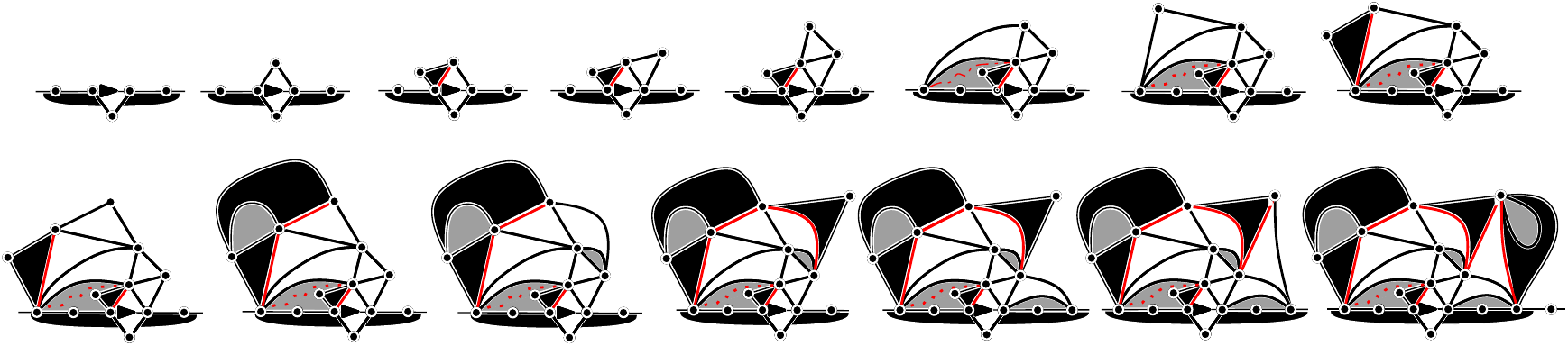}
    \caption{Exploration of the percolation interface in site percolation
      on the dual of a type-1 UIHPT.  The interface is in red.  Note that
      the external face is black, except for one white face added next to
      the root edge.}
    \label{fig:explorationdual}
  \end{figure}

  During this exploration process, we still denote by $(\cE^*_i,
  \cR^*_i)_{i \geq 1}$ respectively the number of exposed edges and the
  number of edges swallowed on the right of $a_i$ in the $i$th peeling
  step. We also let $\epsilon_i=1$ if the face discovered at time $i\geq 1$
  is white and $0$ otherwise.  Hence, the process $(\epsilon_i)_{i \geq 1}$
  is just a sequence of i.i.d.\ Bernoulli variables with parameter $p$, and
  are independent of $(\cE^*_i, \cR^*_i)$.

  Let $S_n$ denote the number of edges on the boundary of $P_n^*$ (or
  equivalently $\bM^*_n$) adjacent to a white face, and let $S_n$ be
  absorbed at $0$.  Then we have $S_0=1$ and as long as $S_{n-1}>0$
  \begin{equation}
    \label{eq:explositedualperco}
    S_{n} = \Big(S_{n-1} - \cR^*_n - 1 \Big)^+ + \epsilon_n \cE^*_n.
  \end{equation}
  Thus the process $(S_n)_{n\geq 0}$ is almost but not quite a random walk with
  i.i.d.\ increments killed at the first hitting time of $\Z^-$.  In
  particular, as long as $S_n$ is above $2$ for triangulations or $3$ for
  quadrangulations, its previous increment is just $\epsilon\cE^*-\cR^*-1$. 
  Still, it is easy to see that $S_n$ will a.s.\ reach $0$ if and only if
  \[
  0 \geq \E[ \epsilon \cE^*- \cR^* - 1]
  = p(1+\delta^*) - \frac{\delta^*}{2} - 1.
  \]
  which completes the proof.
\end{proof}

\textsc{Interface.} In this case also, the above peeling process roughly
follows the leftmost interface of the origin cluster. However, contrary to
site percolation on triangulation some parts of the interface (even before
the last jump) are not explored and are contained in enclosed maps: see the
red dotted line in dark gray parts on \cref{fig:explorationdual}.

\subsection{Bond percolation} \label{sec:bond}

We now turn to bond percolation.  Let us present the setting which is very
similar to the ones treated before.  Let $* \in\allthree$.  To treat bond
percolation a new type of boundary condition will be required.
Conditionally on $\bM^*$ we color the edges of $\bM^*$ with two colors
(black and white) with special boundary conditions: The root edge, and
every edge to its right along the boundary are black.  Every other edge of
the map is colored white with probability $p\in(0,1)$ and black with probability
$1-p$ independently.  Thus the boundary is half-free and half-black.  We
are interested in $\cC = \cC^*_\bond$: the connected white cluster of the
root vertex (the tail of the root edge).

\begin{theorem}\label{thm:bond}
  We have $\P(|\cC^*_\bond| = \infty) > 0$ if and only if $p >
  p^*_{c,\bond}$ where
  \[
  p^*_{c,\bond} = \frac{\delta^*}{2+\delta^*}.
  \]
\end{theorem}

\begin{proof}
  When dealing with bond percolation a new important idea is to keep as
    much randomness as we can.  Thus we do not reveal the status (white or
  black) of all the edges we discover and keep most of them as unknown.
  Instead we only check the color of an edge when necessary to determine if
  this edge is part of $\cC$ or not.  More precisely, we again
  maintain a certain boundary condition on
  $\bM^*_i$.

  One further difference is that we do not reveal a face of the map at
  every step.  Instead, on some steps we only reveal the color of an edge.
  Thus for some $i\geq 1$ we will have that $\bM^*_i=\bM^*_{i-1}$, except
  for differing boundary conditions.
  
  \begin{quote}
    \textsc{Algorithm: } Assume $\bM_i^*$ is a bond-percolated UIHP$*$ with
    a boundary condition of the following form: There is a single finite
    segment of boundary edges that are white (possibly of length $0$).  All
    edges to the right of this segment are black, and all edges to their
    left are unknown and are i.i.d.\ with probability $p$ of being white
    (and also independent of the remaining map and its inside coloring).
    That is we have a free-white-black boundary condition.  We then reveal
    the color of the rightmost unknown edge $a_{i+1}$.  If (and only if) it
    is black, we also perform a peeling step at that edge and reveal a face
    of $\bM^*_i$ \emph{without revealing the status of the new edges
      discovered}.
  \end{quote}
  
  \begin{figure}[h]
    \begin{center}
      \includegraphics[width=16cm]{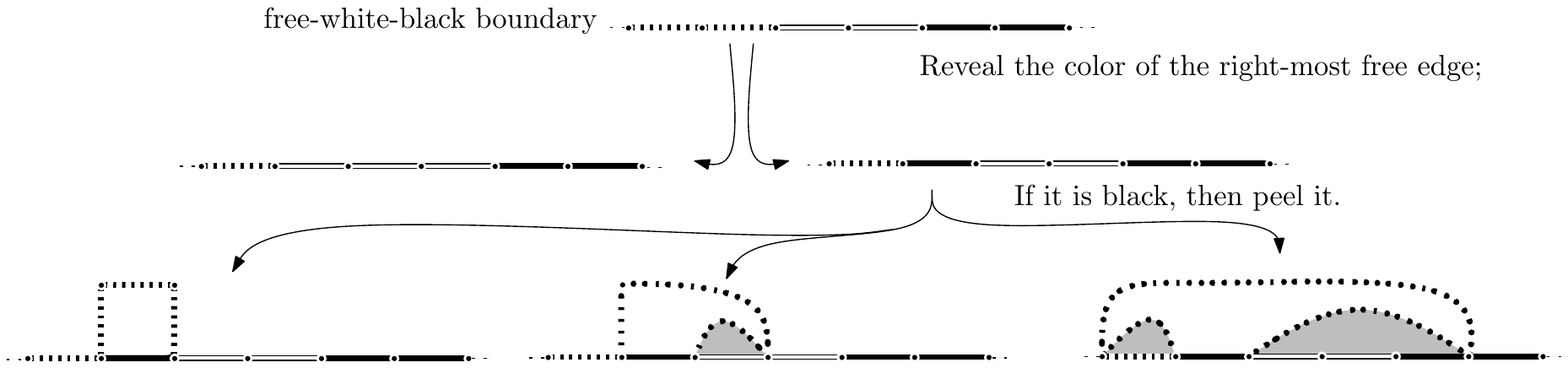}
      \caption{Algorithm for bond percolation. Dotted edges have an unknown
        status.}
      \label{fig:casesbond}
    \end{center}
  \end{figure}
   
  It is again straightforward to check that this form of boundary
  conditions free-white-black is preserved under the peeling process (this
  holds for any class of map) and that the starting boundary condition is
  of the required type (a single origin vertex with $0$ white edges), see
  \cref{fig:casesbond}. We will further assume that this algorithm is used
  from the beginning.
   
  Contrary to the previous cases, there is no clear stopping time for the
  process because there is always a rightmost unknown edge to reveal. In
  fact even if at some peeling time the whole white boundary is eaten there
  is still a possibility that the new vertex at the junction free-black is
  linked to the white cluster, see \cref{fig:stopping}. However as soon as
  $p <1$ it is easy to see that if at time $i\geq 1$ the white boundary is
  of length $0$ then there is a strictly positive probability that at time
  $i+1$ the rightmost unknown edge turns out to be black and the revealed
  face blocks the white cluster that is, even the new junction free-black
  vertex is not part of the white origin cluster. See
  \cref{fig:stopping}. In this case the white cluster of the origin is
  finite.

  \begin{figure}[h]
    \begin{center}
      \includegraphics[width=12cm]{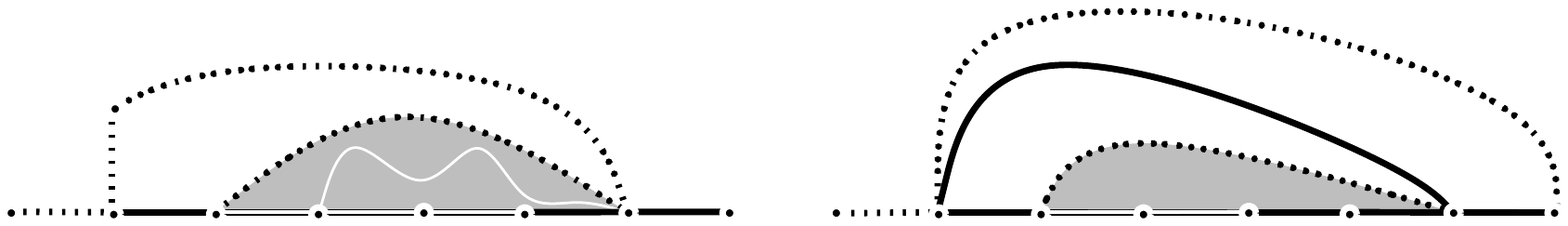}
      \caption{On the left the entire white boundary is swallwed, but it
        is still possible for the white cluster to connect to infinity as
        shown. With positive probability, the next step rules out such a
        connection, as on the right.}
      \label{fig:stopping}
    \end{center}
  \end{figure}

  Let $S_n$ be the number of edges in the white boundary segment of
  $\bM^*_n$, so that initially $S_0=0$.  Let $\epsilon_n$ be the indicator
  of the event that the edge tested in step $n$ is white, and let $\cR^*_n$
  be the number of edges swallowed to the right of $a_n$.  Note that if
  $a_n$ is white then no face is revealed and by convention we let
  $\cR^*_n=0$ in this case.  We do not need $\cE^*_n$ in this model.  Then
  we find that $S_n$ satisfies
  \[
  S_{n} = \Big(S_{n-1} + \epsilon_n - (1-\epsilon_n)\cR^*_n \Big)^+,
  \]
  and is defined for all $n \geq 0$.  According to the previous
  considerations if $S_{n}=0$ infinitely often then $|\cC|<\infty$ almost
  surely. On the other hand if $S_{n}>0$ for all but finitely many $n\geq
  0$, then there is a positive probability that $|\cC| = \infty$ a.s.  The
  expected increment of $S_n$ is $\E[\epsilon_n - (1-\epsilon_n)\cR^*_n] =
  p-(1-p)\delta^*/2$ which is positive precisely when
  $p>\delta^*/(2+\delta^*)$.  The theorem follows.
\end{proof}

\textsc{Interface.} It is useful to consider simultaneously the percolation
configuration on the lattice and the dual percolation configuration on the
dual lattice.  Since each edge of the lattice corresponds to a dual edge in
the dual lattice, the randomness is the same.  We will use the same colors
for an edge and a dual edge, so we study white primary clusters and black
dual clusters.  The above exploration process just follows the (leftmost)
interface of the cluster $\cC$, that is the interface that separates the
cluster of the root vertex from the dual black cluster of the face
containing the edge to its left.  As in \cref{sec:dual_site} not every face
visited by the interface corresponds to a step of the peeling process and
some parts of the interface lie in enclosed maps.  See
\cref{fig:explobondprimal}.

\begin{figure}[!h]
  \begin{center}
    \includegraphics[width=9cm]{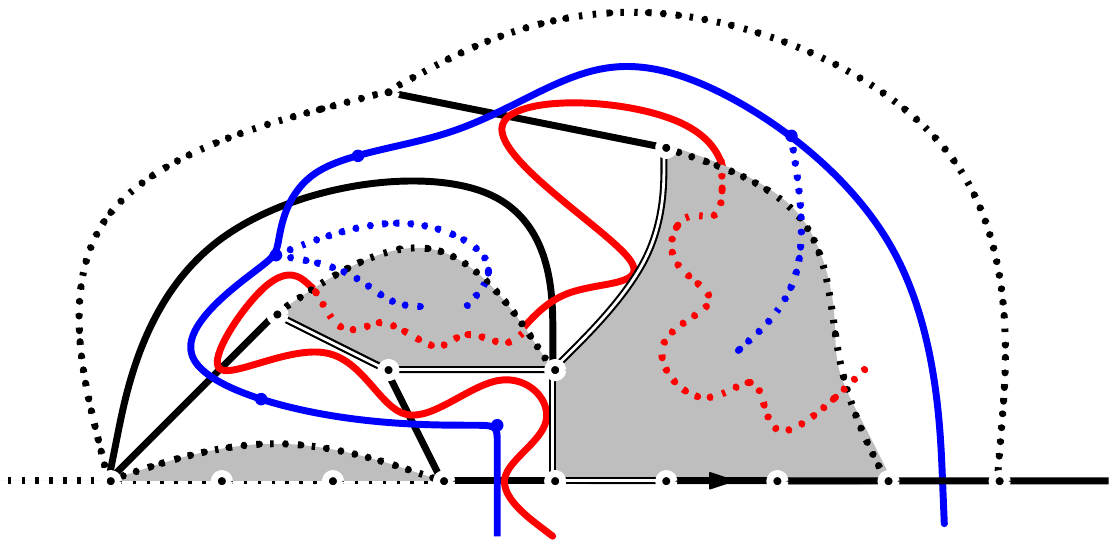}
    \caption{Interface (in red) in bond percolation between the white
      cluster of the origin and the dual black cluster (dual edges in
      blue).}
    \label{fig:explobondprimal}
  \end{center}
\end{figure}

\subsection{Dual percolations} \label{sec:dual_bond}

In this section we study the ``dual'' of the percolations studied in the
last sections.  By dual percolations we mean that if the origin cluster is
blocked this is because there is a ``dual'' cluster in the dual percolation
preventing it from going further.  In particular we shall prove the
unsurprising result that the corresponding thresholds equal $1$ minus the
initial ones.

Since site percolation on triangulations is self-dual this case is already
solved.  As we already noticed, the dual percolation of bond percolation on
the UIHP$*$ is bond percolation on the dual of the lattice.  This process
is studied in \cref{subsec:dualbond}. In \cref{subsec:dualsitedual} we
sketch the analysis of the dual of face percolation which is given by site
percolation on the ``star'' lattice associated to $\bM^*$.

\subsubsection{Dual bond = bond dual} \label{subsec:dualbond}

Here we establish the unsurprising result $p^*_{c,\bond} + p^*_{c,\bond'} =
1$.  While the arguments is very similar to the one used in the previous
section, we describe it here as an illustration of how going from the
primal to the dual lattice and following the same interfaces yields
slightly different exploration procedure.

For sake of clarity we stay with the primal lattice $\bM^*$ but explore the
dual cluster of the origin.  We will first get the result with a slightly
different boundary condition: All the edges of the boundary of the primal
but the root edge are black.  The root edge and all edges not on the
boundary of $\bM^*$ are independent and randomly colored, white with
probability $p\in(0,1)$ and black with probability $1-p$.  Recall that the
color of a dual edge is that of its corresponding primal edge.  We are now
interested in $\cC=\cC^*_{\bond'}$ the dual white bond percolation cluster
containing the dual of the root edge, which by convention is empty if the
root edge is black.

\begin{theorem}\label{thm:bonddual}
  We have $\P(|\cC^*_{\bond'}| = \infty) > 0$ if and only if $p >
  p^*_{c,\bond'}$ where
  \[
  p^*_{c,\bond'} = \frac{2}{2+\delta^*}.
  \]
\end{theorem}

\begin{proof}
 
  As for bond
  percolation, the color of some edges will remain unknown.  At each step
  we shall reveal the color of one edge whose status is still unknown, and
  possibly reveal additionally a face of the map.  The preserved boundary
  condition is now black-free-black.
    
  \begin{quote}
    \textsc{Algorithm:} Assume $\bM_i^*$ is a bond-percolated UIHP$*$ with
    a boundary condition of the following form: All the edges are black
    except a finite connected white region of finitely many edges whose
    unknown colors are i.i.d.\ white with probability $p$ and black with
    probability $1-p$ (and also independent of the unknown region).  We
    then discover the color of the leftmost edge with unknown color. If it
    is black we do nothing more.  If it is white, we also discover the face
    in $\bM^*_i$ adjacent to it but do not reveal the status of the new
    edges.
  \end{quote}
  Clearly $\bM_0^*$ is of the above form and as usual we see that the
  boundary condition black-free-black is preserved under the peeling
  process (see \cref{fig:preserved}), at least as long as unknown edges are
  present on the boundary.  Again, this is valid without any restriction on
  the type of maps considered.
  
  \begin{figure}[!h]
    \begin{center}
      \includegraphics[width=10cm]{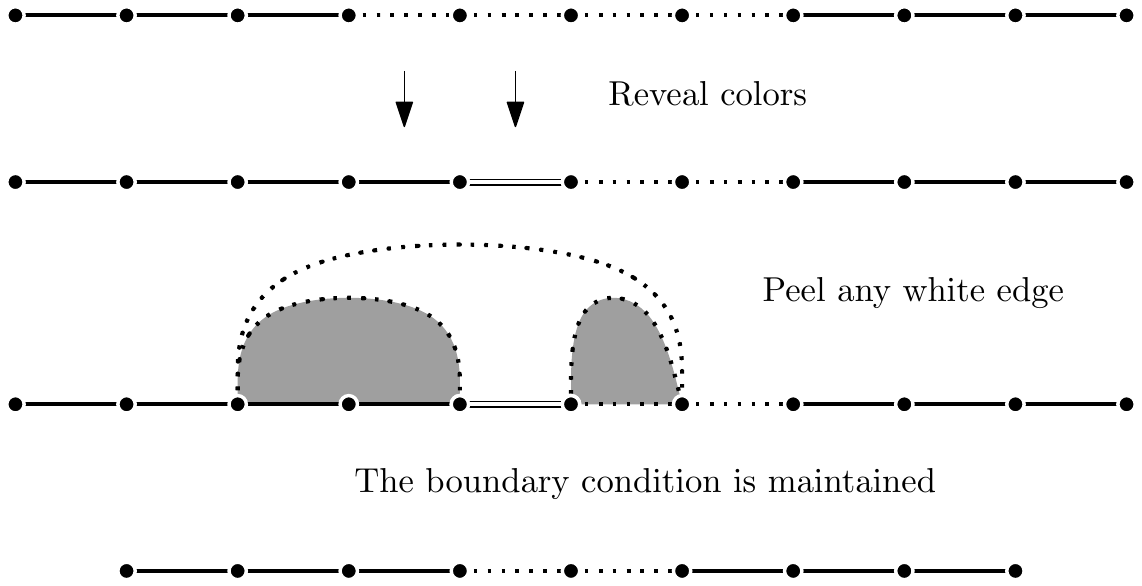}
      \caption{The operations needed to peel an edge.}
      \label{fig:preserved}
    \end{center}
  \end{figure}
 
  If the process is used from the beginning up to time $i>0$, then one can
  check that every edge of unknown color on $\partial\bM^*_i$ that turns out
  to be white is connected in the dual of $P_i^*$ to the dual root edge (if
  it is white).  The process stops when there are no more unknown edges at
  some stage and as usual, a planar topological argument shows that this
  happens precisely when the dual white cluster of the origin edge is
  finite.
 
  Let us denote by $S_n$ the number of edges of unknown color on the
  boundary after $n$ steps of the peeling process.  Let $\epsilon_n$ be the
  indicator of the event that the edge inspected at step $n$ is white, and
  let $\cR^*_n,\cE^*_n$ denote the swallowed and exposed edges as before,
  with the convention that both are $0$ if no face is revealed in the $n$th
  step. Clearly $S_0=1$ and as long as $S_{n-1}>0$ we have
  \[
  S_{n} = \Big(S_{n-1} - 1 - \epsilon_n\cR^*_n \Big)^+ + \epsilon_n \cE^*_n.
  \]
  
  \begin{figure}[!h]
    \begin{center}
      \includegraphics[width=16cm]{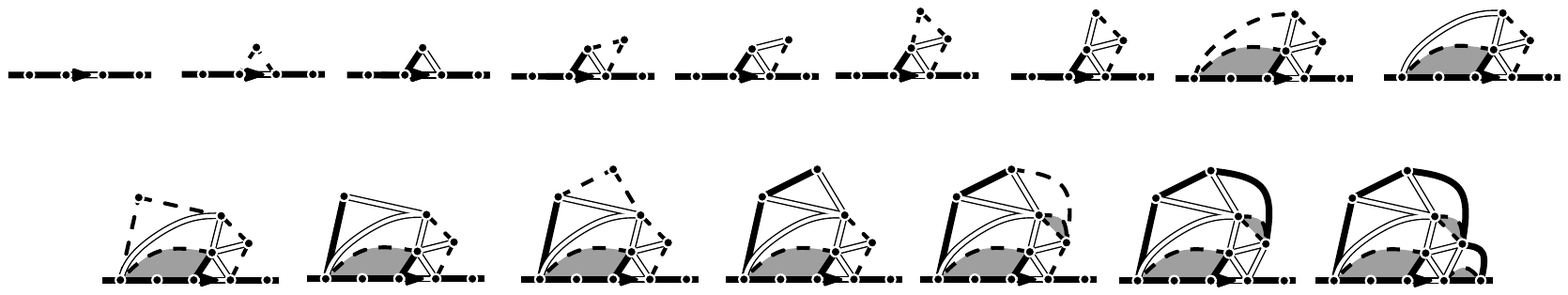}
      \caption{Exploration of dual bond percolation. The
        dotted edges are undetermined at the considered time.}
    \end{center}
  \end{figure}

  As before $(S_n)$ is killed and set to $0$ when reaching $\Z^-$.  Hence
  $(S_n)$ is not exactly a random walk absorbed at $\Z^-$ but the
  increments are independent as long as $S_n$ is large.  In particular we
  see that $S_n$ has a positive probability of remaining positive if and
  only if
  \[
  0 < \E [\epsilon_n (\cE^*_n-\cR^*_n) - 1] = p (1+\delta^*/2) - 1.
  \qedhere
  \]
\end{proof}

\subsubsection{Dual face} \label{subsec:dualsitedual}

Face percolation is not self-dual.  If two faces have only a common vertex
but no common edge, they need not be part of a single connected white
cluster, but if two such faces are black they do form a local barrier for
connection of white faces.  Hence the dual percolation of face percolation
is face percolation but where two faces are declared adjacent if they share
a vertex. Equivalently it corresponds to site percolation on the dual
lattice where we add connections between sites whose dual faces share a
vertex.  This is known as the star-lattice in the case of $\Z^d$. We call it
face$'$ percolation in the sequel. \medskip

We have the expected result 
\[
p_{c, \face'} ^*= 1- p^*_{c,\face}= \frac{\delta^*}{2\delta^*+2}.
\]

Since the reader has already seen several versions of this argument, we
only sketch the proof: The preserved boundary condition is now
black--white--black (that is edges are adjacent to exterior 
faces of those colors). The white part may be empty, in which case we
consider it to include a single vertex.  The peeling rule is to peel at
the edge just to the left of the white revealed part.  The corresponding
recursion for the length of the white boundary is
\[
S_{n} = \Big(S_{n-1} - \cR^*_n\Big)^+ + \epsilon_n\cE^*_n.
\]
As in \cref{sec:bond} even if the white part is swallowed in the process,
it could be that the vertex at the junction enables the origin cluster to
grow further. The problem is treated similarly and as long as $\E[ \epsilon
\cE^* - \cR^*] \leq 0$ we have $S_{n}=0$ infinitely often and then $|\cC| <
\infty$ a.s.  Otherwise $|\cC| = \infty$ with positive probability.  We
leave the details to the interested reader.

\subsection{Free boundary conditions, universality}

In \cref{thm:site,thm:sitedual,thm:bond,thm:bonddual} we focused on the
cluster of a single boundary point (or edge or face) with specially chosen
boundary conditions: black for site percolation on triangular lattice,
black for face, face$'$ and dual bond percolation and free--black for bond
percolation. Note that black boundary condition is the natural setting for
studying the white cluster of the origin in face percolation because of the
presence of the infinite root face which cannot be part of $\cC$
(otherwise, $\cC$ is trivially infinite).  The same remark holds for dual
bond and face$'$ percolation.  However, one can wonder whether free
boundary condition (that is all edges or vertices have i.i.d.\ colors)
changes the percolation threshold in \cref{thm:site,thm:bond}.  The answer
is no: 

\begin{proposition}\label{prop:quenched}
  The percolation thresholds $p^*_{c,\site}$ and $p^*_{c,bond}$ identified
  in \cref{thm:site,thm:bond} correspond to the a.s.\ thresholds for
  percolation on the corresponding percolations on the half-planar maps
  with free boundary conditions.
\end{proposition}

\begin{proof}
  Let us focus on the case of bond percolation with free boundary
  condition. Imagine that we reveal the right neighbors of the root edge
  until we find a black edge. At this point we have a
  free--white--black--free boundary.  We then run our exploration process
  at the free--white junction.  If $p \leq p_{c,\bond}^*$ then
  there will be some time where the white boundary is completely
  swallowed.  At this stage, as in the proof of \cref{thm:bond}, there is a
  positive probability that the next stage of the exploration totally
  blocks the cluster of the origin which is thus finite.  If it is not the
  case we just continue the exploration process at the junction.  We
  eventually end-up with a blocking situation.  Clearly if for every vertex
  $u$ of the boundary the probability that $u$ is in an infinite white
  cluster is $0$ then there is no percolation on the full-map. If $p >
  p_{c,\bond}^*$ there is a positive probability that the cluster
  of the origin is infinite.  By ergodicity of the half planar maps w.r.t.\
  the translation operator (that preserves the map $\bM^*$ but shifts the
  root edge along the boundary) we deduce that in this case there is
  percolation on the map.
  
  The case of site percolation involves similar ideas, although proving
  that there is no percolation at critically is a bit more tedious. We
  safely leave the details to the interested reader.
\end{proof}


\paragraph{Universality.}

Whereas the exploration of site percolation of \cref{thm:site} is specific
to the triangulation case, we have already indicated in the proofs and at
the beginning of \cref{sec:threshold} that the methods developed for bond
and face percolations can be applied to any kind of planar maps without
restriction on the shape of a face.

In \cite{ACpercopeel2} we will prove that the percolation thresholds
identified in this work also correspond to percolation thresholds for
the corresponding models \emph{on the full-plane} UIP$*$.


\section{Critical and off-critical percolation exponents}
\label{sec:exponents}

We now show how peeling along percolation interfaces allows us to deduce
certain geometrical properties of the percolation clusters and in
particular compute certain critical exponents.  For sake of simplicity we
focus on the simplest case which is  site percolation on the triangular
lattice $\two$.  Thus we  fix $*=\two$, and omit it from notation.
We shall comment in \cref{sec:modifications} on the adaptations needed for  
our proofs to cover more general cases.

Recall the setting of \cref{sec:siteperco}: Let $\bM$ be a half-plane UIPT
endowed with Bernoulli site percolation of parameter $p\in(0,1)$ (for
white) with boundary condition given by the infinite boundary being black
with the exception of the root vertex which is white.  \cref{thm:site}
states that the probability that the cluster $\cC:=\cC^*_\site$ containing
the only white vertex of the boundary is infinite is positive if and only
if $p>1/2$.  More precisely, we have from \eqref{eq:explositeperco} that
the length of the white boundary during the exploration process evolves as
a random walk with i.i.d.\ increments of law
\[
\xi^{(p)}_{i} := \epsilon_{i}(\cE_{i}-1) - \cR_{i},
\]
where the joint law of $\cE,\cR$ is given by \cref{prop:mean} and
$\epsilon$ is an independent Bernoulli variable of parameter $p$.  The
process starts at $S_0=1$ and is killed at the first entrance of $\Z^-$.
In particular $\xi^{(p)}$ takes its values in $\{\dots,-2,-1,0,1\}$ and 
satisfies for $k>0$
\[
\P\big(\xi^{(p)} =-k\big) = q_k^\two = \frac{(2k-2)!}{4^k (k-1)!(k+1)!}
\sim \frac{1}{4 \sqrt{\pi}} k^{-5/2}, \qquad \text{as } k \to \infty.
\]

When $p = p_c = 1/2$ the increments $\xi^{(p_{c})}$ have mean $0$.  In this
case we simply write $\xi$ for $\xi^{(p_{c})}$.  Since the r.v.\,$\xi$ is
in the domain of attraction of a stable random variable, the associated
(unkilled) random walk converges once renormalized towards a stable process
of parameter $\frac{3}{2}$.  Let us be a bit more precise and recall some
background about the spectrally negative $\frac{3}{2}$-stable process and
its discrete version.  The interested reader should consult \cite{Ber06}
and the references therein for more details.

We slightly abuse notation here and consider the walk $S$ not killed at the
first entrance of $\Z^-$ that is, let $S_0=1$ and $S_n = 1+\xi_1+ \xi_2 +
\dots + \xi_n$ be a random walk with i.i.d.\ increments of law $\xi$ then
we have
\[
\left(\frac{ S_{\lfloor nt \rfloor }}{n^{2/3}} \right)_{t\geq 0}
\xrightarrow[n\to\infty]{(d)} \kappa \cdot \Big(\cS_t\Big)_{t\geq 0},
\]
where $(\cS_t)$ is the standard $\frac{3}{2}$-stable process with no drift
and no positive jumps with $\kappa = 3^{-2/3}$ and the convergence holds in
distribution for the Skorokhod topology.

By standard spectrally negative $\frac{3}{2}$-stable process, we mean that
its Laplace transform is given by $\E [e^{\lambda\cS_t}] =
e^{t\lambda^{3/2}}$ for all $\lambda \geq 0$, equivalently its Levy measure
is given by
\[
\Pi(dx) = \frac{3}{4 \sqrt{\pi}} |x|^{-5/2} dx \mathbf{1}_{x <0}.
\]
This process is also known as the Airy-stable process and $\cS_1$ as the
(map-)Airy distribution.  Note that the Airy distribution is not symmetric,
has stretched exponential tail on the right but power law tail on the left,
see \cref{fig:airy}.  This process enjoys the scaling property with
parameter $3/2$ that is $(\cS_t)_{t\geq0} = (\lambda ^{-2/3} \cS_{\lambda
  t})_{t\geq0}$ in distribution for any $\lambda>0$.  By the scaling
property, the positivity probability $\P(\cS_t \geq 0)$ is independent of
$t\geq0$ and equals
\[
\rho = \P(\cS_t \geq 0) = \frac{2}{3}.
\]
This quantity is of great importance since it rules the behavior of many
distributional properties of $(\cS_t)_{t\geq0}$ and its discrete analog
$(S_n)_{n\geq0}$.

\begin{figure}[!h]
  \begin{center}
    \includegraphics[width=14cm]{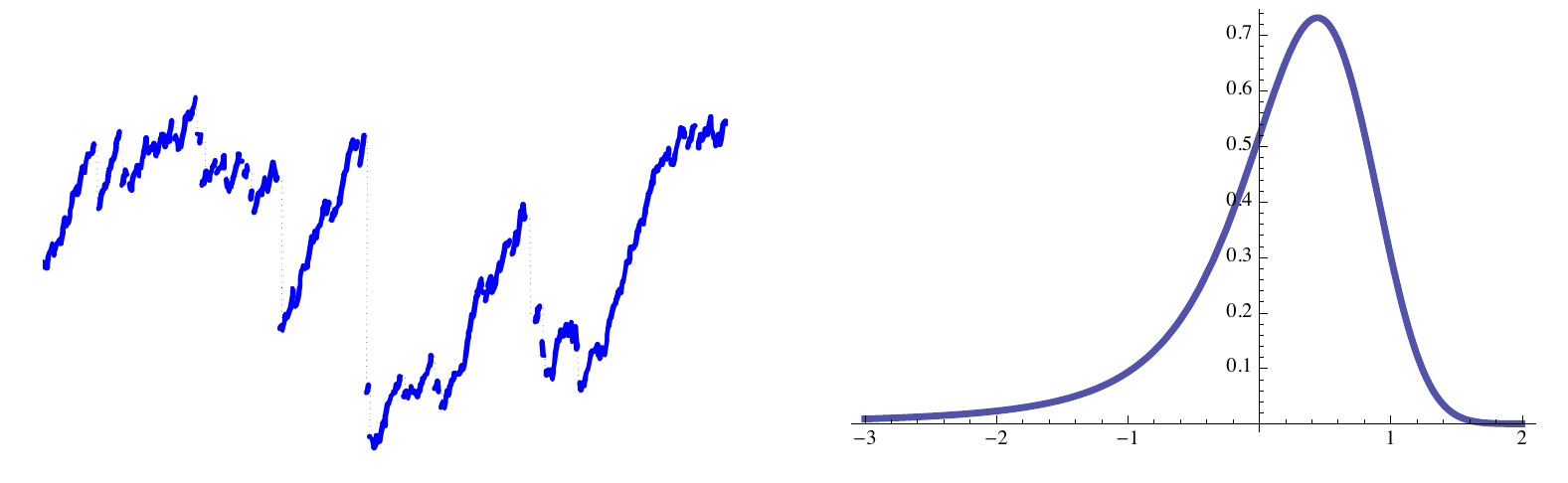}
    \caption{ A sample of an Airy-stable process and the density of its
      distribution at time $t=1$ (also known as the map-Airy
      distribution).}
    \label{fig:airy}
  \end{center}
\end{figure}
\subsection{The off-critical percolation probability}

A first application is the exact computation in $ \mathbf{M}$ of the
probability that the cluster $\cC$ is infinite.

\begin{theorem}\label{thm:exacttheta}
  For site percolation on $\bM^\two$ with black boundary conditions
  except for a white root vertex, we have
  \[
  \P_p(|\cC| = \infty) = \frac{p-\frac{1}{2}}{p} \mathbf{1}_{p > 1/2}.
  \]
\end{theorem}

In particular this implies that the off-critical exponent for the
percolation probability is $1$ (which means $ \mathbb{P}_{p-p_{c}}( |
\mathcal{C}|=\infty) = (p-p_{c})^{1+o(1)}$ for $p \downarrow p_{c}$).  Note
that if the root vertex is instead taken to be random, the percolation
probability becomes simply $\big(p-\frac12\big)^+$.

It is perhaps surprising that the probability of percolation tends to $1/2$
as $p\to1$.  This comes from the fact that even if $p=1$ (that is, every new
discovered point is white) there is still a non-zero probability that
the first step  (or first few steps) of peeling swallows the white boundary
by making a connection towards the right and swallowing the white cluster.

\begin{proof}
  By \eqref{eq:explositeperco} in the proof of \cref{thm:site}, the
  probability that the cluster of the single white point on the boundary of
  $\bM^\two$ is infinite is equal to the probability that a random walk
  starting from $1$ and with i.i.d.\ increments distributed as $\xi^{(p)}$
  never hits $\Z^-$.  In our very specific case this probability can be
  evaluated exactly.  Indeed, if $\xi_i^{(p)}$ are i.i.d. copies of
  $\xi^{(p)}$ we have
  \begin{align*}
    \P_p(1 + \xi^{(p)}_1+ \dots + \xi_n^{(p)} > 0, \forall n \geq 1)
    &= \P( \xi^{(p)}_0 + \dots + \xi_n^{(p)} >0, \forall n \geq 1 \mid
    \xi_0^{(p)} =1) \\
    &= \frac{\P( \xi^{(p)}_0 + \dots + \xi_n^{(p)} >0, \forall n \geq 1)}
    {\P(\xi_0^{(p)}=1)}.
  \end{align*} 
  We have that $\P(\xi^{(p)}=1) = \P(\cE=2,\epsilon=1) = \frac23 p$.  Since
  $\xi^{(p)} \leq 1$, the ballot theorem (see e.g.\ \cite[Theorem
  3]{ABR08}) implies that $\P(\xi^{(p)}_0 + \dots + \xi_n^{(p)} > 0,
  \forall n \geq 1)$ is equal to the mean of $\xi^{(p)}$ which is
  $\frac{2p-1}{3}$.  This finishes the proof of the theorem.
\end{proof}

\subsection{Critical exponents}

We now turn to the study of the geometry of the cluster $\cC$ for site
percolation on $\bM^\two$ at the critical point $p_c=1/2$.  Recall that the
cluster of the only white vertex of the boundary is almost surely finite at
$p_c$.  We denote by $\cH$ the {\bf hull} of $\cC$ that is the sub map of
$\bM$ obtained by filling-in all the finite holes of $\cC$.  This hull has
a connected boundary made of white vertices which we denote by $\partial
\cH$ (unlike the boundary of $\cC$ which may have any number of connected
components).  We also consider the extended hull $\cH^1$ of the cluster
which is made of the hull of all triangles adjacent to $\cC$ (see
\cref{fig:hulls}) or equivalently of the hull of the triangles discovered
during the exploration process (see \cref{fig:perimeter} below). This has
the effect of adding to $\cH$ any so-called fjords, or parts of the
complement connected to infinity only through a single vertex.

We are interested in the boundary of $\cH^1$ w.r.t.\ $\bM^*$, that is the
number $|\partial \cH^1|$ of vertices in $\cH^1$ adjacent to a vertex
outside $\cH^1$.  Note that part of the perimeter of $\cH^1$
coincides with the boundary of $\bM^*$.  That part of the perimeter is not
included in our estimates. Considering the full boundary of 
$\cH^1$ would not change the critical exponents but would require some
additional arguments.  In fact, this part of the perimeter has
the same scale as the rest of the perimeter.  

\bigskip

Before going into the proof of \cref{thm:criticexpo} let us recall a few
facts on random walks in the domain of attraction of a spectrally negative
$\frac{3}{2}$-stable L\'evy process.  Remember that the length of the white
boundary of $\bM_i$ in the exploration of the cluster of the origin evolves
as a random walk $(S_i)_{i\geq 0}$ started from $S_0=1$ and with
independent increments $\xi_{i}= \epsilon_{i}( \cE_{i}-1) - \cR_{i}$ where
the Bernoulli variables $\epsilon_{i}$ have parameter $p_{c}=1/2$. In
particular the negative jumps $S_{i+1}<S_{i}$ correspond to $\cR_{i}
>0$. Recall also that $\xi$ is supported on $\{1, 0, -1, -2,\dots\}$.  The
hitting time of $\Z^-$ for such walks has been analyzed by \cite{VW09} and
(combined with \cite[Theorem 1 (2.4)]{Do82}) we get
 
\begin{lemma}[Hitting time of $ \Z^-$]\label{lemma:T-}
  If $\tau = \inf\{ i \geq 0 : S_{i} \leq 0\}$ is the hitting time of
  $\Z^-$ by $(S_n : n \geq 0)$, then we have
  \[
  \P(\tau = n) \sim c \cdot n^{-4/3} \qquad \mbox{ as }n \to \infty
  \]
  for some $c>0$ (with an explicit but not useful formula).
\end{lemma}

We shall also need an estimate on the fluctuations of sums of i.i.d.\
variables of this type.  This is a standard type of result, but we were not
able to locate a precise reference, so we include a quick proof.

\begin{lemma}[Exponential tail on the right]\label{lem:expotail}
  There exists $c>0$ such that for every $\lambda>0$ and for every $n \geq
  0$ we have for $S_n$ as above
  \begin{align*}
    \P(S_n > \lambda n^{2/3}) & \leq  \exp(-c \lambda).
  \end{align*}
\end{lemma}

In fact, it should be possible to improve the bound to $e^{-c\lambda^3}$,
which would lead to improved powers of $\log n$ in some of the bounds we
get for \cref{thm:criticexpo}.

\begin{proof}
  Using the tail asymptotic of $\xi$ we have for $x>0$
  \[
  \E[e^{x\xi}] - 1  \sim c x^{3/2}, \quad \mbox{as } x \to 0,
  \]
  for some $c>0$.  Applying an exponential Markov inequality after
  multiplying by $x=n^{-2/3}$ we get
  \[
  \P \left( \sum_{i=0}^n \xi^{(i)} > \lambda n^{2/3}\right)
  \leq
  \frac{ \E[\exp(n^{-2/3} \xi)]^n}{e^\lambda} \leq e^{-c+o(1)-\lambda}.
  \qedhere
  \]
\end{proof}
 
\begin{proof}[Proof of \cref{thm:criticexpo}]
  We have sorted $(i)$--$(iii)$ according to the value of the critical
  exponent but we prove first $(ii)$ then $(i)$ and finally $(iii)$. In
  order to lighten notation and spare use the introduction of constants we
  use the symbol $a_{n} \lesssim b_{n}$ if there exists some universal
  constant $C>0$ such that $a_{n} \leq C b_{n}$ for all $n \geq 0$
  
  \paragraph{$(ii)$:  The hull's perimeter.}
  The peeling exploration of the cluster of the origin follows the contour
  of the cluster keeping it on the right.  In particular, the number of
  steps of peeling necessary to explore the cluster of the origin is thus
  the number of triangles adjacent to $\cH$, which is easily related to the
  size of $\partial\cH$ below.  There is a slight problem arising at the
  last step of the exploration.  The exploration reveals enough of the map
  to guarantee that the cluster is finite, but does not enter the region
  surrounded by the last jump, see \cref{fig:perimeter} below.

  \begin{figure}[!h]
    \begin{center}
      \includegraphics[width=14cm]{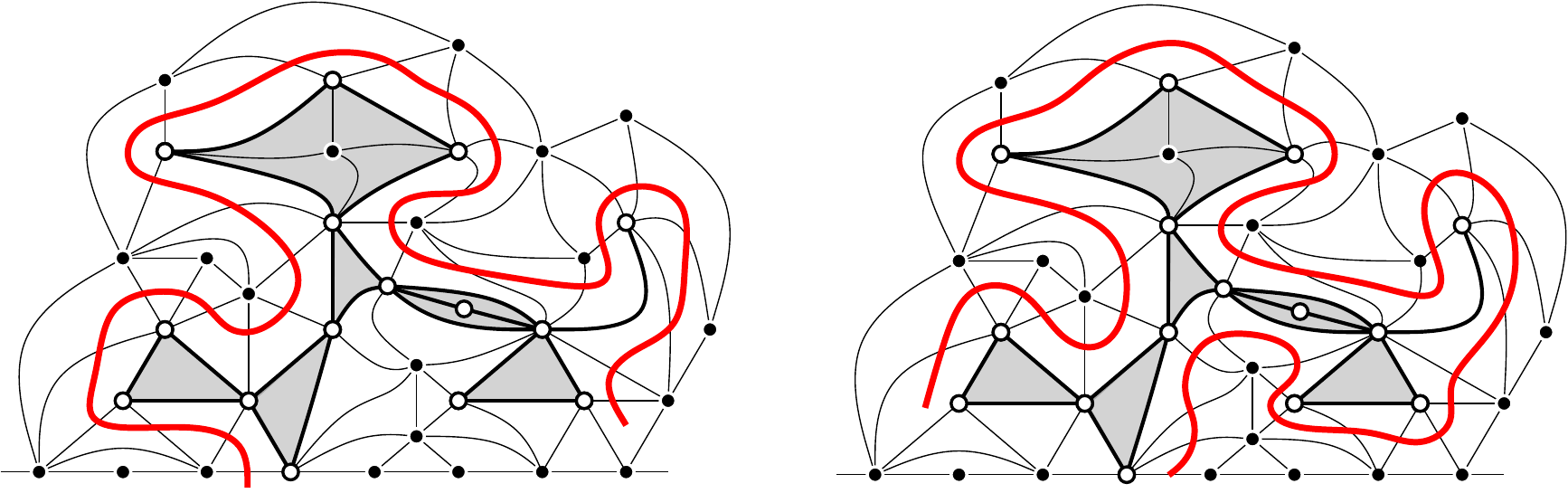}
      \caption{Left: The peeling exploration goes along the boundary of the
        cluster until the first time it touches the axis on the right of
        the root. Right: The interface explored from its other end.}
      \label{fig:perimeter}
    \end{center}
  \end{figure}

  This problem can be circumvented by using a rightmost exploration
  process as depicted in the figure.  This is just a mirror image of the
  process, and since the $\bM$ is symmetric in distribution, this dual
  exploration has the same law as the original one.  Thus if we denote by
  $\tau^\ell$ (resp.\ $\tau^r$) the number of steps of peeling when
  discovering the cluster clockwise (resp.\ counterclockwise) and if
  $|\widetilde{\partial \cH}|$ denotes the number of triangles on the
  boundary of $ \cH$ then we have
  \[
  \tau^\ell \vee \tau^r \leq |\widetilde{\partial \cH}| \leq 
  \tau^\ell + \tau^r.
  \]
  By the description of the peeling process in the triangulation case, both
  $\tau^\ell$ and $\tau^r$ are distributed according as $\tau$ (though they
  are not independent!).  Since each triangle is incident to at most $3$
  vertices of $\partial\cH$, \cref{lemma:T-} now implies that
  $\P(|\partial\cH| > n) \lesssim n^{-1/3}$.

  For the lower bound, observe that at each step of the peeling the
  probability of having $\cR>0$ is some constant ($(1-q_{-1})/2 = 1/6$ as it
  happens).  Thus the number of triangles on the interface incident to any
  vertex of $\partial\cH$ is dominated by a geometric random variable.
  Moreover, these dominating random variables may be made independent.
  (That is, there is a coupling of independent geometric variables and the
  map so that the domination holds a.s.)  By a standard large deviation
  estimate for sums of random variables, for suitable constants,
  \[
  \P(|\partial \cH|>n, \tau^r\vee\tau^\ell < cn) < e^{-cn} \ll n^{-4/3}.
  \]
  With \cref{lemma:T-} this gives the lower bound.

  \paragraph{$(i)$: The hull's volume.}
  The volume of the hull $|\cH|$, can be measured in vertices or faces.  We
  work below with vertices, though the two are directly connected for
  triangulations in terms of the boundary which is smaller by $(ii)$, and
  so the number of faces is of the same order of magnitude.  Faces have a
  slight advantage of in that faces are added to the hull in one way only,
  as the enclosed triangulations encountered on the right of the peeling
  process when we enclose a region and fill it in with a Boltzmann
  triangulation of the proper perimeter, whereas vertices may also be on
  the boundary.  Special care is needed again for the last peeling step
  because the triangulation put inside the last peeling step is not
  entirely contained in the hull.

  Let $Z_i$ be the number of vertices added to the hull in the $i$th step.
  This is $1$ when a new internal vertex is discovered.  When $\cR_i>0$,
  this is the number of internal vertices in the Boltzmann map added,
  except at the last step when only part of that map is in the hull. Thus
  we have
  \[
  \sum_{i=1}^{\tau-1} Z_i \leq |\cH| \leq \sum_{i=1}^{\tau} Z_i.
  \]
  A simple observation using the formulae for $\#\cM^\two_{n,p}$ shows that
  a Boltzmann triangulation with perimeter $p$ typically has a size of order
  $p^2$. More precisely from \cite[Proposition 5.1]{Ang03} and
  \cite[Proposition 6.4]{Ang03} there exists $c>0$ such that for all $p>0$
  and $\lambda > 1$ we have
  \begin{equation}\label{eq:levy}
    \E[Z_{i} | \cR_{i} =p] \sim \frac{2}{3}p^2 \quad \mbox{and} \quad
    \P(Z_{i} > \lambda p^2 | \cR_{i} =p) \leq c \lambda ^{-3/2}
  \end{equation}
  for some $c>0$.  In fact by \cite[Proposition 6.4]{Ang03} the
  distribution of the size of a Boltzmann triangulation of perimeter $p$
  renormalized by $p^2$ converges towards a L\'evy distribution.

  For the lower bound, we exhibit a way for $\cH$ to have size $n$ with
  probability of order $n^{-1/4}$: First the peeling process continues for
  at least $n^{3/4}$ steps that is $ \tau>n^{3/4}$.  This has probability
  of order $n^{-1/4}$ by \cref{lemma:T-}.  On this event, there are
  typically (several) times before $\tau$ when $\cR_i$ is of order
  $(n^{\frac34})^{\frac23} = n^{1/2}$, so with probability bounded from $0$
  there is at least one such jump of size at least $\sqrt{n}$.  Such a jump
  adds to the hull a Boltzmann triangulation of perimeter of order
  $n^{1/2}$ which is typically of size $n$ by \eqref{eq:levy}. Thus on the
  event that $\tau > n^{3/4}$ there is at least a constant probability that
  $|\cH| > n$, and so $\P(|\cH|>n) \geq C n^{-1/4}$ for some $C >0$.

  Informally, the reason this is the typical way of getting a large hull is
  that it is easier to make $\tau$ even larger (exponent $1/3$) than to
  have $\tau$ smaller and have an unusually large Boltzmann map (exponent
  $3/2 > 1/3$).  It is also possible to have $\tau$ small and no abnormally
  large Boltzmann map, but this requires the discrete stable process to
  behave badly, which is even less likely.

  To prove the corresponding upper bound we start by ignoring the last step
  $\tau$, and first split according to $\left\{\tau > n^{3/4}\right\}$:
  \[
  \P\Big(\sum_{i<\tau} Z_i > n\Big)
  \leq \P\Big(\tau>n^{3/4}\Big)
  + \P\Big(\tau\leq n^{3/4}, \sum_{i<\tau} Z_i >n\Big).
  \]
  The first summand is at most a constant times $n^{-1/4}$, so we focus on
  the second and apply Markov's inequality:
  \[
  \P\Big(\tau\leq n^{3/4}, \sum_{i<\tau} Z_i >n\Big)
  = \P\left(\sum_{i=1}^{n^{3/4} \wedge (\tau-1)} Z_i >n \right)
  \leq \frac1n \E \left[\sum_{i=1}^{n^{3/4} \wedge (\tau-1)} Z_i \right].
  \]
  Using \eqref{eq:levy} we have that $\E [Z_i \mid \cR_{i}] \lesssim
  1+\cR_i^2$ (allowing for $\cR_i=0$). So after conditioning
  on $( \cR_{i})$ and taking the $1$ terms outside the sum we end-up with
  \[
  \P\Big(\sum_{i<\tau} Z_i > n\Big)
  \lesssim \frac{1}{n} \left(n^{3/4}
  + \sum_{i=1}^{ \lfloor  n^{3/4}\rfloor} \E \Big[\cR_i^2
  \mathbf{1}_{\tau>i} \Big] \right).
  \]
  To compute the $\E[\cR_i^2]$ we need first to truncate it.  On the event
  $\tau>i$ we have $\cR_i < S_{i-1} \leq i$ (since $S_i$ increases by at most
  $1$ at each step).  By \cref{lem:expotail} we have that $\P(S_i > \lambda
  i^{2/3}) < e^{-c \lambda}$.  Thus with very high probability
  (super-polynomially close to $1$)
  \[
  S_i < (n^{\frac34})^{\frac23} \log^2 n = \sqrt{n} \log^2 n
  \qquad \text{for all } i<n^{3/4}.
  \]
  This allows us to truncate the $\cR_i$s: for some $C>0$ we have
  \[
  \E \left[\cR_i^2 \mathbf{1}_{\tau>i} \right]  \leq
  C + \E \left[\cR_i^2 \mathbf{1}_{\tau>i} \mathbf{1}_{\cR_i < \sqrt{n}
      \log^2 n} \right].
  \]
  
  The next step is to separate the restriction to $\tau > i$.  For this
  we observe that the events $\{\tau > i\}$ and $\{\cR_i>k\}$ are
  negatively correlated since a larger negative jump can only help the
  process hitting $\Z^-$.  Thus conditioning on $\{\tau>i\}$ stochastically
  decreases $\cR_i$ and we have
  \[
  \E \Big[
  \cR_i^2 \mathbf{1}_{\tau>i} \mathbf{1}_{\cR_i < \sqrt{n} \log^2 n} \Big]
  \leq \P\Big(\tau>i\Big)
       \E \Big[\cR_i^2 \mathbf{1}_{\cR_i < \sqrt{n} \log^2 n} \Big]
  \lesssim i^{-1/3} n^{1/4} \log n,
  \]
  where we have used the tail distribution of $\tau$ and the easy estimate
  $\E \left[ \cR^2 \mathbf{1}_{\cR < M} \right] \lesssim \sqrt{M}$.
  Plugging this in we find
  \[
  \P\Big(\sum_{i<\tau} Z_i > n\Big)
  \lesssim \frac{1}{n} \left(n^{3/4}
  + \sum_{i=1}^{ \lfloor  n^{3/4}\rfloor} i^{-1/3} n^{1/4} \log n \right)
  = n^{-1/4+o(1)}.
  \]

  This almost completes the proof.  It remains to show that the
  contribution from the last step when $(S_i)$ hits $\Z^-$ is also unlikely
  to be large. Here we do not have $ \mathcal{R}_{\tau} \leq S_{\tau-1}$
  because of the undershoot below $0$. Let us call $L$ this last jump and
  let $Z$ be the size of the Boltzmann map added during this last jump.
  Again, we may restrict to the event $\{\tau\leq n^{3/4}\}$ and shall
  split according to the value of $\tau$.  By \eqref{eq:levy} we have
  \[
  \P(Z >n \mid L) \leq c\left(\frac{n}{L^2}\right)^{-3/2}.
  \]
  On the other hand we have for $k\geq 1$
  \[
  \P(L = k \mid \tau=i+1, S_i)
  = \mathbf{1}_{k \geq S_{i}} \frac{\P(\xi = -k)}{\P(\xi \leq -S_i)}
  \lesssim S_i^{3/2} k^{-5/2}.
  \]
  Using \cref{lem:expotail} and the bound $S_{i} \leq i+1$ we have $\E[S_i^{3/2} \mid \tau= i+1 ] \lesssim (i+1)$ so  $\P(L = k \mid \tau=i) \lesssim i k^{-5/2}$.  Finally
  \begin{align*}
    \P(Z \geq n)
    &\lesssim \P(\tau \geq n^{3/4}) + \sum_{i=1}^{\lfloor n^{3/4} \rfloor} 
    \P(\tau=i) \left(\sum_{k=1}^{\lfloor \sqrt{n} \rfloor} i k^{-5/2}
      \left(\frac{n}{k^2}\right)^{-3/2} + in^{-3/4} \right) \\
    &\lesssim n^{-1/4}.
  \end{align*}

  \paragraph{$(iii)$: The hull's perimeter, excluding fjords.}
  In order to study $\partial\cH^1$ we consider the evolution of the
  boundary of $\cH^1$ as the peeling process progresses.  A new difficulty
  is that only after the process has terminated we can tell with certainty
  whether any particular vertex is in $\partial\cH^1$ or not.
  %
  As we follow the peeling process, newly revealed black vertices are added
  {\em tentatively} to $\partial\cH^1$, but are removed from it if they are
  in a part of the boundary that is swallowed by a connection to the left
  of the peeling edge.

  In order to control the tail of $\partial\cH^1$, let us introduce an
  auxiliary process $Y_n$, which follows the evolution of the black
  boundary left of $\cC$.  This length is of course infinite always, so
  instead we follow only the change of this boundary, in the form of
  vertices added and removed from the boundary.  Formally let $\cL_i$ be
  the number of edges swallowed {\em on the left} of the $i$th peeling
  point. We have that $(\cE_i,\cL_i)$ have the same law as $(\cE_i,\cR_i)$,
  though $\cL_i$ is not independent of $\cR_i$.  Recall also that
  $\epsilon_i$ is the indicator of the event that this new vertex is white.
  We form the process $(Y)$ by putting $Y_0=0$ and
  \[
  Y_i = Y_{i-1} + \mathbf{1}_{\cE_i=2} (1-\epsilon_i) - \cL_i.
  \]
  Thus $Y$ is just a random walk with i.i.d.\ increments distributed as
  $\xi$.  Note that the two walks $(S)$ and $(Y)$ are not independent,
  since they use the same $\cE$ and $\epsilon$ sequences, and since when
  $\cE=1$, $\cL,\cR$ are not independent.

  With this notation we can determine $|\partial\cH^1|$ from $(Y)$ and
  $\tau$.  Whenever $Y$ reaches its infimum, all the black vertices
  discovered so far are swallowed.  The vertices tentatively in
  $\partial\cH^1$ correspond exactly to increments of $Y$ above its
  infimum.  Thus if we denote $\underline{Y}_i = \min_{j<i} Y_j$, then
  \[
  |\partial\cH^1| = Y_{\tau} - \underline{Y}_{\tau}.
  \]

  While the processes $(Y)$ and $(S)$ (and in particular $\tau$) are not
  precisely independent, they are quite close to independent in the
  following sense.  Recall that we are considering only triangulations for now,
  so either $\cR_i=0$ or $\cL_i=0$ (possibly both).  Since newly discovered
  vertices are also either black or white but not both, each step of the
  peeling process changes at most one of $Y$ and $S$.  It follows that if
  we switch to continuous time with peeling steps controlled by a Poisson
  clock then the two processes become completely independent.  For other
  tyes of maps and percolation models a slightly weaker form of
  independence holds (see \cref{sec:modifications}).

  To get a lower bound on $\P(|\partial\cH^1|>n)$, consider the event
  $\{\tau > n^{3/2}\}$. This event has probability equal to a constant
  times $n^{-1/2}$ by \cref{lemma:T-}.  Conditioned on this event, by the
  above remarks, the process $Y_{\tau} - \underline{Y}_{\tau}$ is roughly
  distributed as $n$ times $\cS_1-\underline{\cS}_1$, where $\cS$ is the
  Airy-stable process introduced in the beginning of the section, and in
  particular the conditional probability of $|\partial\cH^1|>n$ is bounded
  away from $0$.

  Let us prove the corresponding upper bound. Since $Y$ is a sum i.i.d.\
  variables with $0$ mean, bounded above by $1$ and in the domain of
  attraction of a $\frac32$-stable variable, we have from
    \cref{lem:expotail} for any $j<i$ that
  \[
  \P (Y_i-Y_j > n )  \lesssim e^{ -c n / (j-i)^{2/3}} \lesssim e^{-c n/i^{2/3}}
  \]
  and so
  \[
  \P\Big(Y_{\tau} - \underline{Y}_{\tau} > n , \tau=i\Big)
  \lesssim i e^{-c n/i^{2/3}}.
  \]
  For $i<n^{3/2} \log^{-3} n$ the exponential term easily dominates and we
  have that $\P\Big(|\partial\cH^1| > n, \tau \leq n^{3/2} \log^{-3} n
  \Big)$ decays super-polynomially (faster than $n^{3} e^{-c\log^2 n}$).  On the
  other hand, $\P(\tau \geq n^{3/2} \log^{-3} n) \geq C n^{-1/2} \log
  n$, and the proof is finally complete.
\end{proof}

\subsection{Universality of critical exponents} \label{sec:modifications}

In this section we comment on the modifications to make to the previous
section if we consider face or bond percolation on general $\bM^*$.  We
endeavor to address all the different difficulties that arise, and some
possible ways to overcome them.  However, a full proof would involve some
nasty details and so we stay at a rather high and (very) imprecise level,
and formally state \cref{thm:criticexpo} only for triangulations.  May the
reader forgive us.

First, the description of the active boundary $S_n$ in the percolation
exploration process is no longer exactly a random walk killed uppon hitting
$\Z^-$.  However, it is very closely related to such a process.  In
particular, as long as $S_n$ is large the increments are i.i.d.  It is easy
to see that the increments of this walk have mean $0$ exactly at the
critical point, that the increments are bounded from above and have
heavy-tail of exponent $5/2$ on the left (even in the case of
quadrangulations where there may be two segments swallowed on the right, as
in \cref{fig:peelhalfquad3}).  In particular these increments are always in
the domain of attraction of the $\frac{3}{2}$-stable process with no
positive jumps.  Obviously, the value of $\kappa>0$ is changed.

Concerning \cref{thm:exacttheta}, a similar argument holds for some the
other percolation models we discussed, but yields a slightly weaker result.
Indeed, for the other models the associated process $S_n$ has increments
that are bounded but not by $1$.  Thus the ballot theorem does not give the
precise probability of eternal positivity.  Still, the probability that a
random walk on $\Z$ with steps bounded by $k$ with expectation $\mu>0$
remains positive at all times is between $\mu/k$ and $\mu$.  Thus the
identity of the theorem is replaced by lower and upper bounds differing by
a constant.  In the case of bond percolation, the stopping time of the
exploration may not coincide with the first time the active boundary is
swallowed but is lower bounded by the preceding and upper bounded by a
geometric number of them.  All these modifications clearly do not affect the
near-critical exponent so that we still have $\theta(p-p_{c}) =
(p-p_{c})^{1+o(1)}$. 

Since the tail of $\tau$ the stopping time $\tau$ is also going to be the
same, the lower bounds of \cref{thm:criticexpo} are almost unchanged: On
the event $\tau>n$ there is a high probability that $|\cH| > n^{4/3}$, that
$|\partial\cH| > n$ and $|\partial\cH^1| > n^{2/3}$.

For the upper bounds, again most of the estimates still hold.  Some
additional computations arise since it is possible to have multiple
Boltzmann maps revealed at a single step, but mostly the arguments
hold. Special care is needed for the bound on $|\partial\cH^1|$.  Now  the
processes $S$ and $Y$ are even further from being independent, since it is
possible for both of them to make jumps at the same time.  However, it is
possible to show that the large jumps in these processes occur at distinct
times, and so the joint distribution $(Y_n,S_n)_n$ is that of independent
processes.  This allows one to estimate the fluctuations of $Y$ above its
infimum at the killing time $\tau$ of $S$.

  \bibliographystyle{siam}

\end{document}